\documentclass[leqno]{amsart}

\usepackage{natbib,covington,enumerate,setspace, hyperref, amssymb} % Required for inserting images
\usepackage[foot]{amsaddr}

\def\di{\rightarrow}
\def\R{{\mathbb  R}}
\def\N{{\mathbb  N}}
\def\({\textup{(}}
\def\){\textup{)}}

\def\RCA{\textup{\textsf{RCA}}}

\def\ACA{\textup{\textsf{ACA}}}

\newtheorem{definition}{Definition}
\newtheorem{theorem}{Theorem}
\newtheorem{corollary}[theorem]{Corollary}
\newtheorem{proposition}[theorem]{Proposition}

\newtheorem{exa}[theorem]{Example}

\makeatletter
\renewcommand\tagform@[1]{\maketag@@@{\ignorespaces#1\unskip\@@italiccorr}}
\makeatother

\title[From real analysis to the sorites]{\vspace*{-1.5cm} From real analysis to the sorites paradox via reverse mathematics}
\author{Walter Dean}
\author{Sam Sanders}
\begin{document}
%\tableofcontents
\thispagestyle{empty}
\newpage
\begin{abstract}{This paper presents a reverse mathematical analysis of several forms of the sorites paradox.  We first illustrate how traditional formulations are reliant on H\"older's Representation Theorem for ordered Archimedean groups.  While this is provable in $\mathsf{RCA}_0$, we also consider two forms of the sorites which rest on non-constructive principles: the \textsl{continuous sorites} of \citet{Weber2010} and a variant we refer to as the \textsl{covering sorites}.  We show in the setting of second-order arithmetic that the former depends on the existence of suprema and thus on arithmetical comprehension ($\mathsf{ACA}_0$) while the latter depends on the Heine-Borel Theorem and thus on Weak K\"onig's Lemma ($\mathsf{WKL}_0$).  We finally illustrate how recursive counterexamples to these principles provide resolutions to the corresponding paradoxes which can be contrasted with supervaluationist, epistemicist, and constructivist approaches.
}\end{abstract}
\maketitle

%To do

%1) check if the equivalence between COVER and  [LCC_inf + LCC_sup] is helpful in the below. 
%
%Clearly, COVER implies [LCC_inf + LCC_sup], essentially by definition.  
%
%Now assume [LCC_inf + LCC_sup] and fix x_0 such that Phi(x_0).  Suppose that COVER is false ÔupwardÕ (the other direction being symmetric), i.e.\ for all eps>0 there is y in (x_0, x_0 +\eps) with \neg Phi(y).  
%
%Form the set X_0  as {z in R : z\geq x_0 AND \neg Phi(z)}.  Then x_0 = inf X_0 and by LCC_inf, we must have \neg Phi(x_0), a contradiction.  
%
%Thus, using nothing more than Colyvan-Weber, we obtain [LCC_inf + LCC_sup] <=> COVER.  
%

\section{Introduction}\label{S1}

Metaphysics, ethics, philosophy of language, linguistics, psychology, decision theory, and economics all have  interests in the phenomenon of vagueness and the resolution of the sorites paradox.  Within these subjects, several distinct argument forms are also recognized as soritical paradoxes -- e.g.\ the conditional sorites, the mathematical induction sorites, the identity sorites, the line drawing sorites, etc.\footnote{See, e.g., \citep{Oms2019} on both interdisciplinary interest and argument forms.}

This paper will suggest that both lists ought to be extended.  We will do so in the course of providing a mathematical analysis of the so-called \textsl{continuous sorites} as well as formulating a novel form which we will refer to as the \textsl{covering sorites}.  Both forms arise naturally for predicates such as \textsl{short} or \textsl{orange} which are commonly described as applying to \textsl{concrete} or \textsl{perceptual continua}.  

The reasoning of the continuous sorites is mediated by what we will refer to as a \textsl{continuous tolerance principle}.  Such principles attempt to formalize the slogan
\begin{example}
\label{whewell}
\textsl{What is true} up to \textsl{the limit, holds} at \textsl{the limit}.
\end{example}
Principles of this sort figured intermittently in the historical development of real analysis.  Indeed we will see that the mathematical structure of the continuous sorites argument -- which purports to show that if a vague predicate holds at the left endpoint of an interval with such a continuum, then it must hold at its right endpoint -- was repeatedly anticipated from the 17th to the 20th century.   

On the other hand, the argument appears to have only recently been put forth explicitly as a \textsl{paradox} by \citet{Weber2010}.   They follow \citet{Priest2006b} in attributing (\ref{whewell}) to Leibniz and also suggesting that the principle holds for many vague predicates in natural language.   As we will see below, however, it is only in the presence of additional mathematical principles ensuring the completeness of the real numbers -- e.g.\ in the form \textsl{every bounded subset of the reals has a least upper bound} -- that the reasoning in question yields a formal contradiction.   

The covering sorites is similarly mediated by a tolerance principle that ``sufficiently close'' points within continua must either both satisfy a vague predicate or fail to satisfy it.  It too purports to show that if a vague predicate  holds at the left endpoint of an appropriate interval it must hold at its right endpoint.   But we will see that the intermediate reasoning is again mediated by a mathematical principle -- i.e.\ the compactness of the unit interval -- in a manner which was also repeatedly anticipated in early proofs of what we now call the \textsl{Heine-Borel Theorem}.  

We begin our navigation of this material in \S \ref{S2} by considering the setting of the sorites arguments in relation to \textsl{measurement theory} (in the sense of \citealp{Krantz1971}).   We will use this as a basis for reconstructing the continuous sorites in \S \ref{S3} and also for isolating the mathematical principles on which it depends in \S \ref{S4}.  In \S \ref{S5} we show that similar (but mathematically weaker) principles are required to mediate the covering sorites.  In \S \ref{S6} we will use the tools of \textsl{reverse mathematics} (in the sense of \citealp{Simpson2009}) to assess the strength of the principles which are needed to mediate the continuous and the covering sorites.  Our major finding will be that not only do both forms of the argument \textsl{depend} on set existence principles which have been traditionally regarded as non-constructive but their major premises are in fact \textsl{equivalent} to such assumptions (respectively to the systems $\mathsf{ACA}_0$ and $\mathsf{WKL}_0$).    
 
In \S \ref{S7} we refine these results via \textsl{recursive counterexamples} which show how ``gaps'' can arise at what would otherwise correspond to the sharp boundaries of vague predicates.   We will consider how such constructions illuminate supervaluationist, epistemicist, and constructivist approaches to vagueness. In \S\ref{S8}  we briefly indicate how reverse mathematics opens novel avenues for responding to sorites arguments in virtue of the features of the continuum on which they rely.

\section{The sorites and measurement theory}\label{S2}

The idea underlying many sorites arguments is that vague predicates are naturally understood to apply to the elements of structures $A$ describable as either concrete or perceptual \textsl{continua} -- points on concrete lines or transitions between colors or sounds.    These are taken to admit to so-called \textsl{indiscriminable differences} --   a change in the relevant magnitude which cannot be directly perceived or measured -- and are thus regarded as unable to affect whether a point in $A$ satisfies the vague predicate.  But the accumulation of sufficiently many such differences \emph{can} lead to a discriminable difference resulting in a change in the truth value.   

The properties of the structures in question are typically illustrated via examples rather than specified axiomatically -- e.g.\ by taking the vague predicate $\Phi_0(x)$ to be \textsl{short} and $A$ to contain a sequence of objects varying in height or $\Phi_0(x)$ to be \textsl{orange} and $A$ to be a color spectrum.   But of course the possibility of formulating a sorites-like argument for a given predicate will depend on the  principles which the relevant structures are assumed to satisfy.  This is paradigmatically true of the continuous sorites whose formulation is based on the fundamental assumption that
\begin{quote}
\begin{spacing}{.9}
{\footnotesize
[$A$ can be] mapped onto a real-number interval $[0,1]$, exhaustively partitioned into nonempty sets $X = \{x \in [0,1] : \Phi(x)\}$ and $Y = \{x \in [0,1] : \neg \Phi(x)\}$ with $x < y$ for all $x \in X, y \in Y$. \citep[p. 315]{Weber2010}}%
\end{spacing}
\end{quote}
\vspace{-2ex}

Here $\Phi(x)$ is assumed to be the image of $\Phi_0(x)$ in the real numbers under the envisioned mapping -- e.g.\ of the real numbers measuring points on a ruler corresponding to the heights of short individuals or the wavelengths of orange positions in the spectrum.  We thus see from the outset that the formulation of the forms of the sorites which will be relevant here are reliant on representation in a mathematical domain.  But once this is acknowledged, we can also inquire into what must be assumed about the extension of $\Phi_0(x)$ in $A$ to ensure the existence of a function $\phi : A \rightarrow \mathbb{R}$ with properties allowing a given variant of the sorites to go through.   

%%\smallskip

Reflection convinces us that the sorts of predicates which figure in soritical arguments typically come along with an associated order $\prec$ on $A$ -- e.g.\ \textsl{shorter} (or \textsl{strictly less tall}) or \textsl{oranger} (or \textsl{strictly less red}).  Such relations typically figure in so-called \textsl{penumbral connections} (see \citealp{Fine1975}) involving $\Phi_0(x)$ -- e.g.\ that if $a$ is short and $a$ is taller than $b$, then $b$ is also short or symbolically
\begin{example}
$\forall a , b \in A[(\Phi_0(a) \wedge b \prec a) \rightarrow \Phi_0(b)]$
\label{downward}
\end{example}
It is this sort of property which is presumed to be preserved by the aforementioned kind of mapping $\phi$ of $A$ into the real numbers -- e.g.\
\begin{example}
\label{presprec}
$\forall a , b \in A[a \prec b  \rightarrow \phi(a) < \phi(b)]$
\end{example}

It thus becomes apparent that \textsl{measurement theory} in the tradition of \citet{Krantz1971} provides a useful -- and by our lights necessary -- framework for taking into account the relationship between non-mathematical continua $A$ and the real numbers.   The importance of this subject to  theorizing about vagueness is now widely acknowledged.\footnote{See, e.g., \citep[\S 5]{Keefe2000}.}  To articulate the relation  the continuous sorites bears to its better-known discrete counterparts we will review some of its basic conventions.

Measurement theory seeks to provide an account of conditions under which a given non-mathematical structure $\mathcal{A}$ can be represented by some mathematical structure $\mathcal{M}$ such that abstract reasoning about the latter can be used to facilitate drawing concrete conclusions about the former.   Here ``structure'' should be understood in the logician's sense of \textsl{model}.  $\mathcal{A}$ is thus specified as a tuple $\langle A, R_1, \ldots,R_n,  f_1, \ldots, f_m \rangle$ consisting of a domain $A$ relations $R_i$ and functions $f_i$ -- e.g.\ as described by $\prec$ above or the concatenation operation $\circ$ considered below.   These are assumed to satisfy a set of axioms $\Gamma$ which specify the properties of the non-mathematical domain -- e.g.\ that $\prec$ is a linear order or that $\circ$ is associative.   

One of the central aims of measurement theory is to provide an account of when it is possible to establish the existence of a structure preserving mapping, called  a \emph{homomorphism}, from an arbitrary structure $\mathcal{A}$ satisfying $\Gamma$ to a prescribed sort of mathematical structure $\mathcal{M} = \langle \mathbb{M}, S_1, \ldots,S_n,  g_1, \ldots, g_m \rangle$.   In the prototypical cases we will consider,  $\mathbb{M}$ will be the real numbers $\mathbb{R}$ or a subfield thereof and $S_i$ and $f_i$ will thus respectively be relations on $\mathbb{R}^{j_i}$ and functions of type $\mathbb{R}^{k_i} \rightarrow \mathbb{R}$.  A result establishing that there exists a mapping preserving the relevant structure is known as a \textsl{representation theorem} and takes the following form:
\begin{example}
\label{rt}
If the model $\mathcal{A}$ satisfies the axioms $\Gamma$, there exists $\phi:A \rightarrow \mathbb{R}$ such that for all $\vec{a} \in A^{j_i}$, if $R_i(\vec{a})$, then $S_i(\phi(\vec{a}))$ and for all $\vec{a} \in B^{k_i}$, $\phi(f(\vec{a})) = g(\phi(\vec{a}))$. 
\end{example}

Theorizing about vagueness often implicitly assumes the availability of a representation theorem for the predicates in question.  This is paradigmatically so for what we will refer to as \textsl{vague measure predicates} -- i.e.\ adjectives such as \textsl{short},  or \textsl{brief}, \textsl{lightweight} and physical or psychological magnitudes which ``admit to degrees''.\footnote{More broadly the class of relevant terms subsume what linguists refer to as \textsl{gradable adjectives} -- e.g.\ \textsl{slow}, \textsl{cold}, \textsl{orange}, \textsl{sour}, \textsl{difficult}, \textsl{poor}, et cetera.} This is illustrated by the following prototypical soritical argument:\footnote{This is adapted from \citep[p. 7]{Keefe2000}.}

\begin{quote} {\footnotesize
\begin{spacing}{.7}
Intuitively, a millimeter cannot make a difference to whether or not a man counts as tall -- such tiny variations, undetectable using the naked eye and everyday measuring instruments $\ldots$ So we have the principle
\vspace{1ex}
\begin{example}
\label{tolshort}
If $a$ is short, and $b$ is only 1 millimeter taller than $a$, then $b$ is also short.
\end{example}
\vspace{1ex}
But imagine a line of men, starting with someone 1.2 meters tall, and each of the rest a millimeter taller than the man in front of him. Repeated applications of (\ref{tolshort}) $\ldots$ imply that each man we encounter is short, however far we continue. And this yields a conclusion which is clearly false, namely that a man 
greater than 2.2 meters in height, reached after one thousand steps $\ldots$ is also short.
\end{spacing}
}
\end{quote}
\vspace{-2ex}
Several hallmarks of measurement theory are apparent in this and many other similar formulations. This includes a choice of an \textsl{empirical unit} $c$ -- i.e.\ a millimeter -- which is assumed to be less than a so-called \textsl{just-noticeable difference} $d$ for length measurement.  It also includes the specification of a domain $A = \{a_0, \ldots, a_{1000}\}$ of items whose heights are measurable relative to $c$ and which are assumed to be listed in order from shortest to tallest with $a_{i+1}$ being exactly one millimeter taller than $a_i$.  The existence of a measurement function $\phi: A \rightarrow \mathbb{R}^{+}$ is then assumed which preserves the \textsl{shorter} relation $\prec$ in the manner required by (\ref{presprec}) and is such that $\phi(c) = 0.001$ (e.g.\ one millimeter $= \frac{1}{1000}$ meters) and $\phi(c) < \phi(d)$.   

The choice of a specific $c \prec d$, the decision to measure $c$ relative to a particular system of units,  and the choice of a particular unit within that system to represent the length of $c$ are all arbitrary.   But if we wish to formalize the prior passage, then the selection of \textsl{some} specific magnitude and corresponding numerical value for its length is required.    For it is only once such a value has been designated that we are able to state a \textsl{tolerance principle} like (\ref{tolshort}).  This may then be formalized as follows:
\begin{example}
\label{tol1}
$\forall a ,b \in A[(|\phi(a) - \phi(b)| \leq \phi(c) \wedge \Phi_0(a)) \rightarrow \Phi_0(b)]$
\end{example}

To state such a principle also requires that we employ a symbol for the measurement function in our object language description.   This is in fact prototypical of our everyday discourse about various sorts of magnitudes -- e.g.\ when we speak of someone being ``1.2 meters tall''.   The statement of (\ref{tol1}) also shows that the envisioned reasoning is dependent on our ability to refer to ``tiny variations'' to which a vague measure predicate like \textsl{short} is tolerant.  What is assumed here is thus that the domain from which the elements comprising the envisioned sorites sequence are selected contains a sufficiently short object to play the role of a unit.

This in turn illustrates how the \textsl{additive} properties of the relevant mathematical scale also play a role in the envisioned reasoning.   For it is only \textsl{after} such an object has been selected as a unit $c$ shorter than $d$ that we are able to assert that the objects $a_0,\ldots,a_{1000}$ stand in the appropriate relation -- i.e.\ that they are such that the height of object $a_{i+1}$ is the same as that of a compound object resulting from laying $a_i$ end-to-end with $c$.\footnote{\label{archnote1} Descriptions of sorites sequences implicitly assume the existence not only of such a magnitude but also sufficiently many of its ``multiples'' to support the preparation of the relevant sequence of individuals, paint chips, et cetera.  Going forward we thus assume $A$ contains the elements $c, c \circ c, c \circ (c \circ c), \ldots$ whose distinctness is guaranteed by the positivity condition of Theorem \ref{holderthm0}.}   Suppose we introduce `$\circ$' to denote this concatenation operation and `$a \sim b$' to denote the empirical \textsl{same length} relation defined as 
\begin{example}
$a \sim b$ if and only if  \textsl{not} $a \prec b$ and \textsl{not} $b \prec a$
\end{example}
We must thus have that $a_{i+1} \sim a_{i} \circ c$ for all $0 \leq i < 1000$. This is required to ensure that $a_i$ and $a_{i+1}$ stand in the indiscriminability relation and also that $a_{1000}$ has been reached via transitions ``adding up to'' a discriminable difference.\footnote{\label{archnote2} It often seems to be assumed that $\langle \mathcal{A},\prec,\circ \rangle$ is a \textsl{ratio scale} -- i.e.\ the homomorphism between $\mathcal{A}$ and   $\langle \mathbb{R}^+,< , + \rangle$ is determined up to a scalar constant as in the second part of Theorem \ref{holderthm0} below.  If we define $c^0 = c$ and $c^{n+1} = c \circ c^n$, it is then permissible to say that the magnitude of $c^{1000}$ -- i.e.\ a 1 meter magnitude -- is indeed ``$1000$ times as long'' as $c$ itself.   But if we start out by simply declaring that $b \in A$ denotes a \textsl{discriminable} length (i.e.\ $b \succeq d$) -- e.g.\ coinciding with the difference in height between the objects at the beginning and end of a sorites sequence -- then we must also assume that $\mathcal{A}$ satisfies the Archimedean property to ensure that a finite number of copies of $c$ concatenated with the object at the beginning of the series ``add up to or surpass'' $b$.} 

For the soritical derivation to have its intended force, we must thus ensure that $\phi$ not only preserves the structure of the \textsl{shorter} relation $\prec$ in the sense of (\ref{presprec}) but also that of the empirical concatenation operation $\circ$ -- i.e.\ we  additionally require
\begin{example}
\label{prescirc}
$\forall a , b \in A[\phi(a \circ b) = \phi(a) + \phi(b)]$
\end{example}

This illustrates how even familiar forms of the sorites presuppose mathematical representation.\footnote{This point is developed systematically in \citep{Dean2018a} wherein it is argued that both discrete and continuous forms of the sorites implicitly rely on features of mathematical notation (and related Archimedean assumptions) which can be rejected for reasons similar to those discussed in \S\ref{S8}.}  But here the existence of an appropriate measurement function is ensured by a seminal result of measurement theory known as \textsl{H\"older's Theorem}.\footnote{See \citep{Krantz1971} on the significance of this result in the development of measurement theory, its relation to extensive measurement, and the original presentation of \citet{Holder1901a}.}
\begin{theorem} 
\label{holderthm0}
Suppose the structure $\mathcal{A} = \langle A,\prec,\circ \rangle$ satisfies the axioms of an ordered, positive, regular, Archimedean semi-group.   Then for any fixed $c \in A$ and positive real $r \in \mathbb{R}^+$ there exists a homomorphism from $\mathcal{A}$ to $\mathcal{R} = \langle \mathbb{R}^+,<,+\rangle$ mapping $c$ to $r$ -- i.e.\ a function $\phi: A \rightarrow \mathbb{R}^{+}$ satisfying $($\ref{presprec}$)$, $($\ref{prescirc}$)$, and $\phi(c) = r$.   Moreover if $\psi:A \rightarrow \mathbb{R}^+$ is another function satisfying these properties, then there exists $v \in \mathbb{R}^+$ such that $\psi(a) = v \cdot \phi(a)$ for all $a \in A$.   
\end{theorem}
That $\mathcal{A}$ is an ordered semigroup means that $\prec$ is a linear order on $A$,  $\circ$ is associative, and that $x \preceq y$ implies $z \circ x \preceq z \circ y$ and $x \circ z \preceq y \circ z$.  Positivity requires that $x \prec x \circ y$ (and thus there is no identity element). Regularity requires that if $x \prec y$, then there is $z$ such that $x \circ z \preceq y$.  The Archimedean property requires that
\begin{example}\label{arch}
{For all $a \in A$ with $a \prec b$, there is $n \in \mathbb{N}$ such that $b \preceq a \circ \ldots \circ a$ ($n$ times).   }
\end{example}
These conditions are satisfied in cases of \textsl{extensive measurement} which are often in the background of sorites arguments (see \citealp[\S 3]{Krantz1971}).

Such an argument is conventionally formulated as a series of conditionals:
\begin{example}
\label{fm} {\small $\Phi_0(\texttt{a}_0), \Phi_0(\texttt{a}_0) \rightarrow \Phi_0(\texttt{a}_1), \Phi_0(\texttt{a}_1), \ldots, \Phi_0(\texttt{a}_{999}), \Phi_0(\texttt{a}_{999}) \rightarrow \Phi_0(\texttt{a}_{1000}), \Phi_0(\texttt{a}_{1000})$}
\end{example}
Herein the conclusion $\Phi_0(\texttt{a}_{1000})$ is intended to express that the 1000th individual in the envisioned series is \textsl{short}.   But to derive this conclusion we must confirm that each of the conditionals $\Phi_0(\texttt{a}_i) \rightarrow \Phi_0(\texttt{a}_{i+1})$ are true.   This in turn requires not the only the truth of the tolerance principle (\ref{tol1}) but also that each of the empirical statements  $|\phi(\texttt{a}_i) - \phi(\texttt{a}_{i+1})| \leq 0.001$ for $0 \leq i < 1000$ be assumed as  premises of the argument.   But the later statements employ a symbol for the measurement function $\phi:A \rightarrow \mathbb{R}^+$ satisfying (\ref{presprec}) and (\ref{prescirc}) which in turn is only guaranteed to exist in virtue of H\"older's Theorem.   We thus reach the intermediate conclusion -- to which we will return in \S \ref{S6} -- that not only Theorem \ref{holderthm0} but also the mathematical principles required to prove it must be \textsl{counted amongst the premises} of the conditional sorites.  

\section{From the discrete to the continuous sorites}\label{S3}

The foregoing is illustrative of how formulations of traditional forms of the sorites paradox are often implicitly reliant on mathematical representation.   
This is true to an even greater extent in the case of the continuous and covering sorites.   Inspiration for the former derives from the observations that had we elected to formulate (\ref{tol1}) using a \textsl{smaller} choice of the unit $c$ --  e.g.\ a micrometer (i.e.\ $10^{-6}$ meters), a nanometer (i.e.\ $10^{-9}$ meters), etc. -- our confidence in the corresponding version of this form of tolerance principle would presumably \textsl{increase}.  This comes into clearest focus when we examine versions of the sorites  involving not just discrete sequences of objects drawn from non-mathematical continua but the continua  themselves.

%\footnote{Of course our ability to select such objects is premised on the assumption that the objects in which sorites sequences for measure predicates are selected is contained within a larger domain  containing sufficiently small objects with respect to the relevant ordering $\prec$.  What is thus assumed is the existence of a series of objects $e = e_0 \succ e_1 \succ e_2 \succ \ldots$ for which $(e_i)^i \preceq e$ for all $i \in \mathbb{N}$.   This is in fact what is guaranteed by the regularity requirement on $\mathcal{A}$ in Theorem \ref{holderthm0} which is a weak form of the ``no minimal element'' assumptions employed in traditional formulations of H\"older's Theorem.  Such an assumption also seems implicit in most philosophical treatments of sorites argues for predicates like \textsl{short}.  This comes into clearest focus when we examine versions of the sorites  involving not just discrete sequence of objects drawn from non-mathematical continua but from various sorts of continua  themselves.   

%%\smallskip

These are exemplified by predicates describing temporally or spatially extended regions which are said to have ``fuzzy boundaries'' -- e.g.\ a period of time asserted to contain \textsl{the moments of  mid-morning} or a region asserted to contain the \textsl{points on Mount Everest}. A related class of cases concerns what are often explicitly referred to as \textsl{perceptual continua} -- e.g.\ a crescendo from \textsl{quiet} to \textsl{loud tones} or a spectrum containing wavelengths varying from a paradigm shade of \textsl{orange} to one of \textsl{red}. 

In such cases the empirical or perceptual domain $A$ may be assumed to contain a subinterval $C = \{c \in A : a \preceq c \preceq b\}$  consisting of the points ``between'' the endpoints $a,b \in A$ in the relevant sense -- e.g.\ the points on a line between the summit of Mount Everest and a sufficiently distant point or between shades in an orange-red spectrum.   In such cases we typically think that the predicate in question holds of the left endpoint -- i.e.\ $\Phi_0(a)$ -- and fails to hold of the right endpoint -- i.e.\ $\neg \Phi_0(b)$.   But such a description does not provide us with a sequence of objects from $C$ by which we can discretely step through the sort of reasoning described above.

To formulate a sorites argument in such cases, it thus seems that we have little choice but to assume the existence of an appropriate measurement function $\phi:A \rightarrow \mathbb{R}$ for which $\phi(a) = u$ and $\phi(b) = v$.  In such cases it is also typically assumed that the image of $C$ under $\phi$ is a closed interval $[u,v] := \{x \in \mathbb{R} : u \leq x \leq v\}$ in the real numbers.\footnote{E.g.\ ``Imagine a patch darkening continuously from white to black. At each moment during the process the patch is darker than it was at any earlier moment. Darkness comes in degrees. The patch is dark to a greater degree than it was a second before, even if the difference is too small to be discriminable by the naked eye. Given that there are as many moments in the interval of time as there are real numbers between 0 and 1, there are at least as many degrees of darkness as there are real numbers $[\in [0,1]]$, an uncountable infinity of them.'' \citep[p. 113]{Williamson1994}}   But although we may also feel that $\Phi_0(x)$ is insensitive to small variations, we will often be unaware of a relevant magnitude in $d \in A$ to serve as a just noticeable difference or of its numerical image.\footnote{One possible means of circumventing this is to posit the existence of \textsl{some} magnitude which we assume is ``smaller'' than $d$ in the relevant sense.  The corresponding version of (\ref{tol1}) would take the form $(\exists e \in A)( \forall a , b \in A)[(|\phi(a) - \phi(b)| \leq \phi(e) \wedge \Phi_0(a)) \rightarrow \Phi_0(b)]$.    However this is not sufficient to derive a contradiction from the other soritical assumptions in virtue of the existence of nonstandard models in which the existence of such a magnitude is witnessed by an infinitesimal.}  The question thus arises how we can formulate a tolerance principle which can be regarded as part of the meaning $\Phi_0(x)$ in a manner which still gives rise to a deductively mediated contradiction.   

One way of avoiding the artificiality of choosing a specific $c$ is to maintain that what we in fact accept is that if $\Phi_0(x)$ is vague and holds of $a \in A$, then $\Phi_0(x)$ also holds of points $b \in A$ which are ``vanishingly close'' to $a$.  Or we may abandon the metrical connotations of ``close'' by asserting that if $\Phi_0(x)$ is vague and $b$ is within a ``neighborhood'' of a point $a$ satisfying $\Phi_0(x)$, then $b$ also satisfies $\Phi_0(x)$.  

The terms \textsl{neighborhood} and cognates like \textsl{region} or \textsl{vicinity} evidently express topological concepts.  But it is still not obvious how to formulate a tolerance principle which takes advantage of any of these notions.  To do so, Weber \& Colyvan rely on the assumption that the real numbers satisfy \textsl{Dedekind completeness} -- i.e.\ 
\begin{itemize}
\label{dc}
\item[DC:] Suppose $X \subseteq \mathbb{R}$,  $X \neq \emptyset$ and $X$ has an \textsl{upper bound} -- i.e.\ there is $y \in \mathbb{R}$ such that $\forall x(x \in U \rightarrow x \leq y)$.  Then there exists a $u \in \mathbb{R}$ which is the \textsl{supremum} (or \textsl{least upper bound}) of $X$ -- i.e.\ $\forall x(x \in U \rightarrow x \leq u)$ and if $v$ an is upper bound of $X$, then $u \leq v$.   In this case we write $u = \sup(X)$.  
\end{itemize}
In general a bounded set need not contain its supremum. But if $u = \sup(X)$ then $u$ is as close as possible to points in $X$ in the metrical sense -- i.e.\ $\forall \varepsilon > 0 \exists x \in X(|\sup(X)-x| < \varepsilon)$.  And any neighborhood of $u$ in the Euclidean topology on $\mathbb{R}$ -- i.e.\ $(s,t) := \{x \in \mathbb{R} : s < x < t\}$ with $u \in (s,t)$ -- will also contains points in $X$.  

%%\smallskip

In order to formulate a tolerance principle which exploits this property of the real numbers, Weber \& Colyvan presuppose that the field $A$ of $\Phi_0(x)$ has been mapped in the manner we have been discussing into a closed subinterval of $\mathbb{R}$ which they take to be $[0,1]$.     Their proposed tolerance principle is then formulated for a mathematical predicate $\Phi(x)$ which is understood as the image of $\Phi_0(x)$ in $\mathbb{R}$ under a suitable measurement function $\phi(x)$ -- i.e.\ 
\begin{example}
\label{defPhi}
$\Phi(x) := \exists a \in A(\Phi_0(a) \wedge \phi(a) = x)$.
\end{example}
Writing  $\mathrm{BA}(X)$ to abbreviate that a set $X$ is \textsl{bounded above}, we can now state Weber \& Colyvan's proposed tolerance principle as follows:\footnote{This formulation differs from that given in \citep[p. 315]{Weber2010} by making the condition that $X$ has a least upper bound explicit in the antecedent (which in addition to $X \neq \emptyset$ is required to ensure that $\sup(X)$ is a denoting term).  In initially assuming that $X$ varies over subsets of $\mathbb{R}$ rather than sequences we have also followed the treatment in \citep[p. 243]{Weber2021}.}
\begin{equation*}
\label{lcc0}
\tag{$\mathrm{LCC}_{\sup}(\Psi):$}  \forall X \subseteq \mathbb{R}(X \neq \emptyset \ \& \ \mathrm{BA}(X) \ \& \ \forall x(x \in X \rightarrow \Psi(x)) \rightarrow \Psi(\sup(X)))
\end{equation*} 
Weber \& Colyvan refer to this as the \textsl{Leibniz Continuity Condition} [LCC] and suggest that it holds for $\Phi(x)$ obtained via (\ref{defPhi}) from a vague predicate $\Phi_0(x)$.  We will have more to say about its motivation in \S\ref{S4}.  But for the moment we will focus on reconstructing its role in mediating a sorites-like contradiction.  

%%\smallskip

Suppose  $\mathcal{A}$ satisfies the hypotheses of Theorem \ref{holderthm0} and also the following hold:
\begin{example}
\label{contconds}
\begin{enumerate}[i)]
\item $a,b \in A$ and $a \prec b$
\item There exists $\phi:A \rightarrow \mathbb{R}^+$ satisfying the conclusion of Theorem \ref{holderthm0} such that $\phi(a) = 0$ and $\phi(b) = 1$ and also that $\phi$ is \textsl{onto} $[0,1]$ -- i.e.\ $\forall x \in [0,1]\exists a \in A (\phi(a) = x)$.
\item $\Phi_0(a)$ and thus $\Phi(0)$
\item $\neg \Phi_0(b)$ and thus $\neg \Phi(1)$
\end{enumerate}
\end{example}
Inspection of the proof of H\"older's Theorem ensures that $\phi(x)$ can also be constructed to satisfy the first part of (\ref{contconds}ii).\footnote{That $\phi$ is onto $[0,1]$ is \textsl{not} guaranteed by Theorem \ref{holderthm0}.   The existence of such a measurement function must thus be taken as an additional premise of the argument.   But this appears to be Weber \& Colyvan's intention as they describe $A$ as having been ``mapped \textsl{onto} $\ldots$ $[0,1]$'' such that it is ``\textsl{exhaustively} partitioned'' by $\Phi(x)$ in the manner of \ref{csarg}iii \citeyearpar[p. 315, our emphasis]{Weber2010}.}   Part i), iii), and iv) are traditional soritical premises.    The goal is then to use $\mathrm{LCC}_{\sup}(\Phi)$ to show that if $\Phi(x)$ holds of $0$ then it holds of $1$.  But unlike (\ref{tol1}), $\mathrm{LCC}_{\sup}(\Phi)$ does not allow us to extend the set of points satisfying $\Phi(x)$ rightward by any fixed increment.  Thus despite the similarity in form between these principles, there is no apparent way of harnessing this principle to mimic the reasoning (\ref{fm}) of the conditional sorites.

There are indeed some subtleties in selecting principles sufficient to derive a contradiction from the premises (\ref{contconds}) in conjunction with $\mathrm{LCC}_{\sup}(\Psi)$. One complication is that the argument presented by Weber \& Colyvan requires not just  $\mathrm{LCC}_{\sup}(\Psi)$ but also an analogous principle about lower bounds:\footnote{In the traditional development of analysis it is possible to show that the existence of suprema implies the existence of infima (see, e.g., \citealp[p. 5]{Rudin1976}).   On the other hand, there is no apparent way of \textsl{deriving} $\mathrm{LCC}_{\inf}$ from $\mathrm{LCC}_{\sup}$ absent other assumptions.}
\begin{equation*}
\label{lcc1}
\tag{$\mathrm{LCC}_{\inf}(\Psi)$:}
\forall X \subseteq \mathbb{R}(X \neq \emptyset \ \& \ \mathrm{BB}(X)  \ \& \ \forall x(x \in X \rightarrow \Psi(x)) \rightarrow \Psi(\inf(X)))
\end{equation*} 
Here $\mathrm{BB}(X)$ abbreviates that $X$ has a lower bound -- i.e.\ $\exists y(\forall z \in X \rightarrow y \leq z)$ -- and $\inf(X)$ denotes its \textsl{infimum} (i.e.\ \textsl{greatest lower bound}).

%%\smallskip

Adopting this assumption their argument can now be set out as follows:\footnote{This reconstructs \citep[p. 315]{Weber2010} where the reasoning is attributed to an unpublished paper of James Chase.  \citep[pp. 243-244]{Weber2021} also presents a formulation based on a paraconsistent axiomatization of analysis to which we will briefly return in note \ref{paranote}.}  
\begin{example}
\label{csarg}
\begin{enumerate}[i)]
\item Assume $\mathrm{LCC}_{\sup}(\Phi)$ and $\mathrm{LCC}_{\inf}(\Phi)$ for a vague measure predicate $\Phi$.
\item Assume for a contradiction that $\neg \Phi(1)$.
\item Define the following subsets of real numbers 
\begin{center} $U = \{x \in [0,1] : \Phi(x)\}$ and $V = \{x \in [0,1] : \neg \Phi(x)\}$ 
\end{center}
\item We obtain the following from i), property (\ref{downward}) of $\prec$, and (\ref{contconds}i-iii):
\begin{enumerate}[a)]
\item $U \neq \emptyset$ and $U$ is bounded above.
\item $V \neq \emptyset$ and $V$ is bounded below.
\item $U \cup V = [0,1]$ and $U \cap V = \emptyset$.
\end{enumerate}
\item By iv a), DC entails that $\sup(U) := u$ exists.  Thus $\Phi(u)$ by the definition of $U$ and $\mathrm{LCC}_{\sup}$ applied to $\Phi(x)$
\item By iv b), DC entails that $\inf(V) := v$ also exists.   Thus $\neg \Phi(v)$ by the definition of $V$ and $\mathrm{LCC}_{\inf}$ applied to $\neg \Phi(x)$
\item By the trichotomy property of $<$ on $\mathbb{R}$ one of the following must hold: $u = v$; $u < v$; $v < u$.  We now reason by cases as follows:
\begin{enumerate}[a)]
\item If $u = v$, then $\Phi(u)$ and $\neg \Phi(u)$ by v) and vi).   Contradiction.   
\item If $u < v$, then by the density of $<$ on $\mathbb{R}$, there exists $w$ such that $u < w < v$.   Since $w < v$, $w \in U$ and thus $\Phi(w)$.   But since $u < w$, $w \in V$ and thus $\neg \Phi(w)$.   Contradiction.
\item If $v < u$, a contradiction in reached analogously with b).
\end{enumerate}
\item Discharging the reductio assumption ii), we conclude $\Phi(1)$.   
\end{enumerate}
\end{example}
The foregoing constitutes a demonstration of $\Phi(1)$ from the premises (\ref{contconds}i-iii), the two forms of the LCC and Dedekind Completeness.  This then yields a contradiction with the other soritical assumption $\neg \Phi(1)$ (premise \ref{contconds}iv).   It is this overarching reasoning which we henceforth refer to as the \textsl{continuous sorites argument}.\footnote{Weber \& Colyvan also present a related argument which they refer to as \textsl{topological sorites} which is based on the \textsl{connectedness} of the spaces in question rather than their completeness.   While they ultimately suggest that this is a yet more basic form of the paradox, we will not treat it explicitly here for three reasons.  First \textsl{pace} Weber \& Colyvan, we are dubious that vague predicates in natural language come  with topologies which are given independently of metrics induced via the techniques of measurement theory.  Second, even without this assumption (and using a using the pre-topological notion of a \textsl{vicinity space}) \cite{Dzhafarov2019} shows that an argument similar to Weber \& Colyvan's requires a principle equivalent to $\mathsf{ACA}_0$ over $\mathsf{RCA}_0$ in context of reverse mathematics.   We will present a related third consideration in note \ref{topnote2} below. \label{topnote}}

\section{The Leibniz Continuity Condition and Open Induction}\label{S4}

We have presented the continuous sorites argument as showing that $\Phi(1)$ follows from $\Phi(0)$ and the two forms of the LCC.  But a similar argument may be used to show that these assumptions entail that \textsl{all} real numbers $x \in \mathbb{R}$ fall under $\Phi(x)$. Indeed in analogy to mathematical induction on $\mathbb{N}$ we have the following:   
\begin{proposition}
\label{doi}
Let $\Psi(x)$ be a predicate of real numbers and suppose that $\{x\in \R: \Psi(x)\}$ is a non-empty open set in the Euclidean topology on $\mathbb{R}$ -- i.e.\ $\exists x \Psi(x)$ and
\begin{example}
\label{open}
$\forall x(\Psi(x) \rightarrow \exists \varepsilon > 0 \forall y(x - \varepsilon < y < x + \varepsilon \rightarrow \Psi(y)))$
\end{example}
Suppose also that $\Psi(x)$ is closed downward -- i.e.\ 
\begin{example}
\label{down} 
$\forall x (\forall y(\Psi(y) \ \wedge \ x < y) \rightarrow \Psi(x))$
\end{example}
-- and satisfies 
\begin{example}
\label{progr} 
$\forall x (\forall y(y < x \rightarrow \Psi(y)) \rightarrow \Psi(x))$
\end{example}
Then $\Psi(x)$ holds for all real numbers -- i.e.\ $\forall x \in \mathbb{R}\hspace{.4ex} \Psi(x)$.
\end{proposition}
\begin{proof} Suppose that $\Psi(x)$ satisfies (\ref{open}), (\ref{down}) and (\ref{progr}), let $U = \{x \in \mathbb{R} : \Psi(x)\}$ and $v \in U$.   Since $\Psi(x)$ satisfies (\ref{down}), $\Psi(w)$ holds for all $w < v$.  Thus $U$ has no lower bound.   Suppose for a contradiction that $U$ had an upper bound.   Then since $U \neq \emptyset$, $\sup(U)=u$ exists by DC.  Since $\forall y(y < u \rightarrow \Psi(y))$, $\Psi(u)$ follows by (\ref{progr}).   But now note that since $U$ is open and $u \in U$, there exists $\varepsilon \in \mathbb{R}$ such that $(u - \varepsilon, u + \varepsilon) \subset U$.    If $0 < \delta < \varepsilon$, then $u + \delta \in U$.  But since $u + \delta >u$, this contradicts $u = \sup(U)$. 
\end{proof}

As we will discuss further in \S\ref{S5}, the foregoing style of argument was repeatedly anticipated in the development of real analysis from at least the late 19th century.\footnote{A number of 20th century rediscoveries are  documented by \citep{Clark2019}.}   One of the latter chapters was the isolation of a principle  which \citet[pp.~69-70]{Lorenzen1971} called \textsl{analytical induction}.  This is simply the assertion that if $\Psi(x)$ satisfies the hypotheses of Proposition \ref{doi} then it also satisfies its conclusion.\footnote{Lorenzen originally observed that this principle can often be used to replace the assumption 
of Dedekind Completeness for \textsl{arbitrary subsets} of $\mathbb{R}$ -- understood in contrast to those defined by a delimited class of predicates -- within a predicative development of analysis.}  Writing $\mathrm{Open}(\Psi)$ to abbreviate (\ref{open}), $\mathrm{Down}(\Psi)$ for (\ref{down}), and $\mathrm{Prog}(\Psi)$ (or \textsl{Progressive}) for (\ref{progr}),  this can be formulated schematically as follows:
\begin{equation}
\tag{$\mathrm{OI}_0(\Psi):$}  \exists x \Psi(x) \ \wedge \mathrm{Open}(\Psi) \wedge \mathrm{Down}(\Psi) \wedge \mathrm{Prog}(\Psi)) \rightarrow \forall x \in \mathbb{R}\hspace{.4ex} \Psi(x)
\end{equation}

The proof just given for this principle is similar in many respects to the continuous sorites argument.    Better conformity to the reasoning of (\ref{csarg}) can be obtained by applying a version of OI$_0(\Psi)$ restricted to the unit interval $[0,1]$  to which \citet{Coquand1992} gave the name \textsl{Open Induction}.  Using the abbreviation 
%\smallskip
\begin{equation*}
\tag{$\mathrm{Prog}^{[0,1]}(\Psi):$}
\label{prog01} 
 \forall x \in [0,1]\big((\forall y \in [0,1])(y < x \rightarrow \Psi(y)) \rightarrow \Psi(x)\big)
\end{equation*}
this takes the form\footnote{Note that $\Psi(0)$ follows from  $\mathrm{Prog}^{[0,1]}(\Psi)$.  Thus if we are only concerned with the behavior of $\Psi(x)$ on this interval,  the conclusion follows by reasoning about an arbitrary $x \in [0,1]$ in the same manner as the preceding proof without the need to additionally assume $\mathrm{Down}(\Psi)$.}  
\begin{equation*}
\tag{$\mathrm{OI}_1(\Psi)$:}
(\mathrm{Open}(\Psi) \ \wedge \ \mathrm{Prog}^{[0,1]}(\Psi)) \rightarrow (\forall x \in [0,1]) \Psi(x)
\end{equation*}

In these principles the assumption that $\Psi(x)$ is progressive plays a role similar to LCC$_{\sup}(\Psi)$.   Thus even better conformity to the continuous sorites can be had by considering a restricted version of the former
%\smallskip
\begin{equation*}
\tag{$\mathrm{LCC}^{[0,1]}_{\sup}(\Psi):$}
\forall X \subseteq [0,1](X \neq \emptyset \ \wedge  \forall x(x \in X \rightarrow \Psi(x))) \rightarrow \Psi(\sup(X)))
\end{equation*}
%\smallskip
to formulate another version of Open Induction as follows:
%\smallskip
\begin{equation*}
\tag{$\mathrm{OI}_2(\Psi)$:}
  (\mathrm{Open}(\Psi) \ \wedge \ \Psi(0) \ \wedge  \mathrm{LCC}^{[0,1]}_{\sup}(\Psi)) \rightarrow (\forall x \in [0,1]) \Psi(x)
\end{equation*}

A disadvantage of this formulation is that it quantifies over  subsets of the reals and as such is a \textsl{third-order} statement.   But Weber \& Colyvan also discuss a form of the LCC in which the variable $X$ is understood to vary not over subsets of $\mathbb{R}$ but rather over Cauchy sequences. 
%CHECK are the sequences monotone?  otherwise, the claim about sup in the next sentence is not correct. 
%RECHECK
We can also thus consider 
\begin{itemize}
\item[$\mathrm{LCC}^{[0,1]}_\mathrm{seq}(\Psi):$]
For any $x\in [0,1]$, if for any sequence $(x_{n})_{n\in \N}$ in $[0, x]$ with $\lim_{n\di \infty}x_{n}=x$,\\ if we have $\Psi(x_{n})$ for all $n\in \N$, then $\Psi(x)$.
\end{itemize}
and the corresponding version of Open Induction
\begin{equation*}
\tag{$\mathrm{OI}_3(\Psi)$:}  (\mathrm{Open}(\Psi) \ \wedge \ \Psi(0) \ \wedge \ \mathrm{LCC}^{[0,1]}_\mathrm{seq}(\Psi))  \rightarrow (\forall x \in [0,1]) \Psi(x).
\end{equation*}

We may now formulate a derivation similar to the continuous sorites argument which consolidates the reasoning of steps (\ref{csarg}.iii)-viii) into a single application of $\mathrm{OI}_{i}$ for $j \in \{1,2,3\}$.  Retaining the basic assumptions (\ref{contconds}), we have:
\begin{itemize}
\item[(\ref{csarg}$'$)]
\begin{enumerate}[i)]
\item Suppose $\mathrm{Open}(\Phi)$, $\Phi(0)$, and $\mathrm{Prog}^{[0,1]}(\Phi)$ (for $j = 1$),  $\mathrm{LCC}^{[0,1]}_{\sup}(\Phi)$ (for $j = 2$) or $\mathrm{LCC}^{[0,1]}_\mathrm{seq}(\Psi)$ (for $j = 3$).
\item Suppose $\mathrm{OI}_j(\Phi)$.
\item By i) and ii), $\forall x \in [0,1]\Phi(x)$ and thus, in particular, $\Phi(1)$.
\end{enumerate}
\end{itemize}
As with (\ref{csarg}), this yields a contradiction with the further assumption $\neg \Phi(1)$.

In comparing the derivations (\ref{csarg}) and (\ref{csarg}$'$), it should be kept in mind that $\mathrm{OI}_{1}(\Psi), \mathrm{OI}_2(\Psi)$ and $\mathrm{OI}_3(\Psi)$ are all \textsl{theorem schemas} of classical real analysis.  Just as mathematical induction is taken as a valid schema for deriving statements about natural numbers, these principles may all be regarded as valid schema for deriving statements about the unit interval.

The validity of mathematical induction is often regarded as \textsl{constitutive} of the structure $\mathbb{N}$.  But it can also be regarded as a consequence of some but not all definitions of the natural numbers.   Similarly, the proof of Open Induction is related to the assumption that the real numbers satisfy Dedekind Completeness which, in turn, plays a central role in the traditional ``construction of the reals'' in terms of Dedekind cuts.  But as we will discuss in \S\ref{S6}, there are other definitions of $\mathbb{R}$ which lack this property.  We will also see that in such settings there is a sense in which each of OI$_0$-OI$_3$ are  \textsl{equivalent} to Dedekind Completeness.   

Presuming that such principles are retained, we can now also see that the assumptions which the continuous sorites brings into conflict are the openness of the extension of $\Phi(x)$ and that $\Phi(x)$ satisfies either Progressiveness or one of the latter two forms of the LCC.\footnote{The role of the assumption $\mathrm{Open}(\Phi)$ is obscured  in Weber \& Colyvan's original presentation of the continuous sorites which relies  on \textsl{both} $\mathrm{LLC}_{\sup}(\Phi)$ and $\mathrm{LLC}_{\sup}(\Phi)$ but does not explicitly require that the extension of $\Phi(x)$ is open.  Note, however, that when the sets $U$ and $V$ are defined at step iii) of the argument (\ref{csarg}) it must be the case that either $U = [0,u)$ and $V = [u,1]$ or  $U = [0,u]$ and $V = (u,1]$ -- i.e.\ exactly one of them is \textsl{half open} which is in fact sufficient for $\mathrm{OI}_1,\mathrm{OI}_2,\mathrm{OI}_3$ to hold.   Weber \& Colyvan later come close to acknowledging this in adopting a condition similar to $\mathrm{Open}(\Phi)$ as part of their definition of vagueness -- i.e if $\Phi(x)$ is the image of a vague predicate and has extension $X \subseteq [0,1]$, then for each $x \in X$, there exists a neighborhood $U_x$ containing $x$ on which the characteristic function of $X$ is constant \citeyearpar[p. 318]{Weber2010}.}  As we will discuss further in \S\ref{S5}, the role of  $\mathrm{Open}(\Phi)$ in the  proof of Open Induction (in any of its forms) may itself seem reminiscent of an intuitively plausible kind of tolerance principle extending the extension of $\Phi(x)$ from $u \in [0,1]$ to an open set $(u - \varepsilon, u + \varepsilon)$ surrounding it.  But in the continuous setting this on its own is not sufficient to infer $\Phi(1)$ from $\Phi(0)$ as a bounded open sets need not contain its least upper bounds.  It is such a transference of $\Phi(x)$ from the points in $X$ to $\sup(X)$ which the various forms of the LCC are intended to effect.  So we must finally attend to their motivation.

%%\smallskip

The expression ``Leibniz Continuity Condition'' is taken from \citet[\S 11.4]{Priest2006b} who in turn attributes a principle similar to $\mathrm{LCC}_{\sup}$ to Leibniz.   Leibniz referred to the principle in question as his \textsl{Law of Continuity} which he first published in 1687.   One of his more succinct formulations was as follows:
\begin{quote}
\begin{spacing}{.9}
{\footnotesize 
If any continuous transition is proposed terminating in a certain limit, then it is possible to form a general reasoning, which covers also the final limit. \\ \hspace*{1ex} \hfill \citep[p. 147]{Leibniz1920a}
}
\end{spacing}
\end{quote}
\vspace{-3ex}
Leibniz intended this principle to express part of his conviction that physical transitions had to be continuous in a manner analogous to geometrical limiting processes. It is in virtue of this interpretation that Leibniz's principle was later sloganized as ``what is true up to the limit is true at the limit'' (see, e.g., \citealp{Cajori1923}).  

It is tempting to read the sequential formulation $\mathrm{LCC}^{[0,1]}_\mathrm{seq}(\Psi)$ of the LCC into such historical sources.  But such principles are not in general true for arbitrary predicates of real numbers.\footnote{For instance suppose that $x_n = \frac{1}{1^3} + \frac{1}{2^3} +  \frac{1}{3^3} + \ldots \frac{1}{n^3}$  and that $\Psi(x)$ is the predicate \textsl{$x$ is rational}.  Then $ (x_n)_{n\in \mathbb{N}}$ is an increasing sequence such that $\Psi(x_n)$ as each partial sum $\sum_{i = 1}^{n} \frac{1}{i^3}$ is rational.   But although  $(x_{n})_{n \in \mathbb{N}}$ is convergent -- and thus $\lim_{n \rightarrow \infty} x_n := v \in \mathbb{R}$  exists -- it can also be shown that  $\neg \Psi(v)$ as $v$ is irrational.  This is a result of \citet{Apery1979}.}   While \citet[p. 316]{Weber2010} flag their awareness of this, they still assert that such a principle is \textsl{constitutive} of how we should understand vague predicates.   And indeed debate about the Law of Continuity continued to play a role in the development of the calculus during the 18th century.  But far from confirming the principle, the discovery that modeling more complex kinds of physical dynamical systems required the recognition of broader classes of discontinuous functions was taken  as a decisive \textsl{refutation}.\footnote{The initial steps are reconstructed by \citet[\S 33]{Truesdell1960} in regard to a debate between d'Alembert and Euler about the equations which can be employed to model the initial state of a plucked string.  About this Truesdell writes ``Euler's refutation of Leibniz's law was \textsl{the greatest advance in scientific methodology} in the entire [18th] century'' (his italics).}  It thus seems that the ideas behind the LCC stand in bad company alongside an outmoded understanding of both and mathematical and concrete continua.  

%%\smallskip

%We will suggest in \S\ref{S7} that the bearing of this on the semantics of vague predicates is illuminated by \textsl{recursive counterexamples} to Open Induction.  But it will first be useful to present a related argument which does not depend on the completeness assumptions which we will see in \S\ref{S6} are latent in the LCC.

%; the latter are latent in the various forms of Open Induction and the Leibniz Continuity Condition, as shown in \S 6.

\section{The Covering Sorites and the Heine-Borel Theorem}\label{S5}

The forms of the sorites we have considered thus far rely on tolerance principles which push the extension of the image $\Phi(x)$ of a vague predicate $\Phi_0(x)$  rightward from $0$ so that it ultimately surpasses some designated endpoint.  The variant we will consider in this section is grounded in the \textsl{prima facie} weaker assumption  that each point $a$ in an empirical or perceptual interval  is surrounded by some region  $U^{a} = \{b \in A : a \prec b \circ c_a \ \& \ b \prec a \circ c_a\}$ of points which are sufficiently ``nearby'' $a$ such that all points closer than the ``tiny variation'' $c_a$ agree with $a$ on $\Phi_0(x)$.   

If this is true of \textsl{every} $a \in A$, it follows that there must be some means of partitioning such a continuum into a not-necessarily disjoint family $\mathcal{D} = \{U_i : i \in I\}$ of such regions with the following properties: 
\begin{example}
\label{cov0}
\begin{enumerate}[i)]
\item Every $a \in A$ is contained in some region $U_i$.
\item The regions are sufficiently ``small''  such that all $b,c \in U_i$, $\Phi_0(b) \leftrightarrow \Phi_0(c)$.
\end{enumerate}
\end{example}

The intuition motivating the existence of such a family $\mathcal{D}$ is similar to that for $\mathrm{Open}(\Phi)$.  Its is supported by the relationship between vagueness, margins of error arising from finite sensitivity, and the use of open sets to model them.\footnote{See, e.g., \citep[\S 8]{Williamson1994}.} The plausibility is bolstered further by observing that a family witnessing (\ref{cov0}i,ii) might consist of infinitely or even uncountably many regions of sufficiently ``fine grain''.

%%\smallskip

The tolerance principle on which the covering sorites is based results from using the framework of \S \ref{S2} to formalize the properties (\ref{cov0}i,ii).  In this case regions $U$ in $A$ go over to open sets  $O = \{z \in \mathbb{R}^+ : \phi(c) - \varepsilon_a < z < \phi(a) + \varepsilon_a\}$ in $\mathbb{R}$ for a suitable $\varepsilon_a = \phi(c_a)$.  We assume that $\phi:A \rightarrow \mathbb{R}^+$ is a bijection between $A$ and $[0,1]$ satisfying the properties of Theorem \ref{holderthm0}.  In this case the conditions (\ref{cov0}) go over to the claim that there exists a family of intervals $\mathcal{C} = \{O_i : i \in I\}$ satisfying 
\begin{example}
\label{cov}
\begin{enumerate}[i)]
\item $\forall x \in [0,1] \exists i(x \in O_i)$ 
\item $\forall i \in I \forall x,y \in O_i(\Psi(x) \leftrightarrow \Psi(y))$
\end{enumerate}
\end{example}
We will write $\exists \mathcal{C}\mathrm{Cov}(\Psi, \mathcal{C})$ to abbreviate that a cover satisfying (\ref{cov}i,ii) exists.  

%%\smallskip

The covering sorites seeks to derive a contradiction from the same soritical premises (\ref{contconds}) as the continuous sorites by relying on the compactness of the Euclidean topology on $[0,1]$ rather than Dedekind Completeness or one its equivalents.   The former is reported by what is now called the \textsl{Heine-Borel Theorem}:
\begin{theorem}
\label{hb}
Suppose that $\mathcal{C} = \{O_i : i \in I\}$ is an open covering of $[0,1]$ $($or in general any closed and bounded set $X$ of real numbers$)$.   Then there are $i_{0}, \dots, i_{k}\in I$ such that for any $x\in [0,1]$, there is $j\leq k$ with $x\in O_{i_{j}}$, i.e.\ $\{O_{i_{j}}: j\leq k \}$ is a finite sub-covering of $[0,1]$ $($or $X)$. 

%a family of sets $\mathcal{C}_0 \subseteq \mathcal{C}$ with $\mathcal{C}_0 = \{O_j : j \in J\}$ for a finite set $J \subseteq I$ such that for all $x \in [0,1]$, there exists some $j \in J$ such that $x \in O_j$.
%NEW a =>  some
\end{theorem}
We will return to the proof of Theorem \ref{hb} in a moment.   But to see how it can be used to formulate a soritical derivation, we record the following consequence:
\begin{proposition}
\label{prop2}
Suppose that there exists an open $\Psi$-covering of $[0,1]$ and $\Psi(0)$.  Then $\forall x \in [0,1] \Psi(x)$.    
\end{proposition}

While this is immediate from Theorem \ref{hb}, it is still useful to isolate Proposition~\ref{prop2} as a corollary as it bears a similar relation to the Heine-Borel Theorem as does Proposition \ref{doi} (and thus Open Induction)  to Dedekind Completeness.  In particular, we may use it to formulate a schema encapsulating how the compactness of the unit interval may be used to uniformly mediate a general inference:
\begin{equation*}
\tag{$\mathrm{CL}(\Psi):$}
(\exists \mathcal{C}\mathrm{Cov}(\Psi,\mathcal{C} ) \ \wedge \ \Psi(0)) \rightarrow \forall x \in [0,1] \Psi(x)
\end{equation*}
We will call this principle the \textsl{Creeping Lemma}.\footnote{The name was suggested for a similar principle by \citet{Moss1968a}. The Creeping Lemma is also cited by \citet{Clark2019} as one of several anticipations of Open Induction.  However we will see in \S 6 that it is a strictly weaker principle from a logical point of view.}

%%\smallskip

The covering sorites for a vague predicate $\Phi(x)$ is now as follows:\footnote{We have been unable to locate any  anticipations of this argument.  The closest approximations of which we are aware are the \textsl{topological sorites} of \citet{Weber2010} and the \textsl{classificatory sorites} of \citet{Dzhafarov2019}. But see notes \ref{topnote} and \ref{topnote2}.} 
\begin{example}
\label{comparg}
\begin{enumerate}[i)]
\item Suppose $\exists \mathcal{C}\mathrm{Cov}(\mathcal{C},\Phi)$ and $\Phi(0)$.  
\item Suppose $\mathrm{CL}(\Phi)$.  
\item By i) and ii), $\forall x \in [0,1]\Phi(x)$ and thus, in particular, $\Phi(1)$.
\end{enumerate}
\end{example}
Again as in the case of derivation (\ref{csarg}), this reasoning yields a contradiction with the further soritical assumption $\neg \Phi(1)$.  But the Creeping Lemma is again a theorem schema of classical analysis.     As with the LCC, whether one regards the above as a genuine paradox will thus likely covary with the stock one places in $\exists \mathcal{C}\mathrm{Cov}(\mathcal{C},\Phi)$ for a given predicate.  But there is still a connection between the continuous sorites and the manner in which the  Heine-Borel Theorem was originally obtained.

The Heine-Borel Theorem can be credited directly to \citet{Borel1895} in the case of countable coverings.  Both his original proof as well as several others which followed exemplify two salient features.\footnote{See \citep{Dugac1989} on the history and early proofs of the Heine-Borel Theorem.} First, assuming that $\mathcal{C}$ is a covering of $[0,1]$, they employ sorites-like arguments to show that it is possible to derive a contradiction from the assumption that there is a point in $x \in [0,1]$ which cannot be reached via a finite sequence of overlapping intervals from $\mathcal{C}$ the first of which contains $0$.  Second, in order to do so, they employ principles equivalent to Dedekind Completeness.

Both points are  illustrated by Lebesgue's  \citeyearpar[pp. 104-105]{Lebesgue1904} proof of Theorem~\ref{hb}:
\begin{proof}
Let $\mathcal{C} = \{O_i : i \in I\}$ be an open covering of $[0,1]$ and suppose for reductio that there does not exist a finite sub-cover of $\mathcal{C}$.  Say a point is \textsl{reached} if it satisfies
\begin{equation*}
\text{$\Psi(x) = $ \textsl{the interval $[0,x]$ can be covered by a finite number of intervals from $\mathcal{C}$}}
\end{equation*}

%\smallskip

\noindent Our reductio assumption implies that the set $W = \{x \in [0,1] : \neg \Psi(x)\} \neq \emptyset$.   But since $\mathcal{C}$ covers $[0,1]$ it follows that $0 \not\in W$ since $0$ is reached as witnessed by some $O_{i_0} = (x_{i_0},y_{i_0})$. As $W \neq \emptyset$ and is bounded below,  $\inf(W) = w$ exists by DC. But then $w \in O_k = (x_k,y_k)$ for some $k \in I$.  By the density of the reals there also exist $u,v$ such that $x_k < u < w < v < y_k$.  By the definition of $W$, $u$ is reached.  There thus exists a finite sub-cover $\mathcal{C}^-_0 \subseteq \mathcal{C}$ of $[0,u]$.   But then $\mathcal{C}_0 = \mathcal{C}^-_0 \cup \{(x_k,y_k)\}$ is a finite covering of $[0,v]$, contradicting the definition of $w$ as $\inf(W)$.     
\end{proof}
The method of this proof is similar to that of Open Induction which is in turn similar to the continuous sorites argument itself. The  development of the topology of reals in the 19th century thus anticipated the covering sorites in much the same way that reflection on limits in the 18th century anticipates the continuous sorites.  But since the proof of Theorem \ref{hb} just given also relies on Dedekind Completeness, one might conclude that the covering and continuous sorites share a common mathematical core.   As will see, however, argument (\ref{comparg}) requires \textsl{strictly weaker} assumptions than does (\ref{csarg}). It thus it seems reasonable to regard the continuous and covering sorites as genuinely \textsl{distinct} forms of the paradox.

\section{Reverse Mathematics of the sorites}\label{S6}

In this section we use the methods of reverse mathematics to show that the central assumption underlying the continuous sorites is equivalent to a non-constructive principle known as \textsl{Arithmetical Comprehension} while that underlying the covering sorites is equivalent to a strictly weaker (but still non-constructive)  principle known as \textsl{Weak K\"onig's Lemma}.  This will inform our discussion in \S \ref{S7} .  

\subsection{Subsystems of second-order arithmetic}
\label{S6.1}

Reverse mathematics is a subfield of mathematical logic devoted to determining the set existence principles necessary to prove theorems of ordinary, non-set theoretic mathematics, inclusive of real analysis. This is done by formalizing the theorems in question in the \textsl{language of second-order arithmetic} $L_2$, and then proving their equivalences to axiomatic systems located in a canonical hierarchy of set existence principles. The proofs are carried out in a weak base theory known as $\mathsf{RCA}_0$, which can be understood as a codification of computable mathematics.  The canonical presentation is the monograph \citep{Simpson2009} whose conventions we will briefly review.\footnote{See \citep{Eastaugh2024} for an introduction targeted at philosophers and \cite{Dzhafarov2022} for a recent textbook.}

The language $L_2$ extends the language of first-order arithmetic -- which has variables $n,m, \ldots$ intended to range over natural numbers along with non-logical symbols $0,+,\times,<$ -- with second-order variables $X,Y,\ldots$ intended to range over \textsl{sets} of natural numbers.   The axiomatic system $\mathsf{Z}_2$ of \textsl{full second-order arithmetic} has following components: 1) the axioms of $\mathsf{PA}^-$ (i.e.\ a slight strengthening of Peano arithmetic without induction); 2) the second-order induction axiom $\forall X((0 \in X \ \wedge \ \forall n(n \in X \rightarrow n+1 \in X)) \rightarrow \forall n(n \in X))$; 3) the full second-order comprehension schema $\exists X \forall n(n \in X \leftrightarrow \varphi(n))$ where $\varphi(n)$ is a $L_2$-formula not containing $X$ free.  The \textsl{standard model} of $L_2$ is $\mathcal{N} = \langle \mathbb{N}, \mathcal{P}(\mathbb{N}), 0,+,\times,< \rangle$ in which the first- and second-order variables respectively range over $\mathbb{N}$ and the power set of $\mathbb{N}$.

Despite its apparent limitations, many other sorts of objects can be represented (or \textit{coded}) in $L_2$.    For instance, a positive \emph{rational number} $q \in \mathbb{Q}^{+}$ can be coded by a pair $(m, n)$ of natural numbers which $q=\frac{m}{n}, n \neq 0$ and $m,n$ are relatively prime using a definable pairing function.   Elements of $\mathbb{Q}$ can then be coded by triples to keep track of signs.  In turn, a \emph{real number} $x$ can be coded by a sequence of rationals $q_0, q_1, q_2, \ldots$ with an exponential rate of convergence 
\begin{example}
\label{mod2k}
$\forall k \forall i(|q_k - q_{k+1}| \leq 2^{-k})$.
\end{example}
Since rational numbers are coded by individual natural numbers, in this way a given $x \in \mathbb{R}$ corresponds to a set $X \subseteq \mathbb{N}$ as in the traditional arithmetization of analysis.  Similarly, a sequence of reals $x_0, x_1, x_2, \ldots$ can then be represented as a doubly indexed sequence of rationals, which in turn may be understood as a function (which may itself be coded as a set) of type $F:\mathbb{N} \times \mathbb{N} \rightarrow \mathbb{Q}$ such that $x_j$ is coded by $q_{0,j},q_{1,j}, q_{2,j},\ldots$ with $q_{i,j} := F(i,j)$.  Finally an \textsl{open set} $O$ of real numbers can be coded by two sets of rationals $\{q_{0}, q_{1}, q_{2}, \dots\}$ and $\{q'_{0}, q'_{1}, q'_{2}, \dots\}$ where we view $O$ as represented by the union $\bigcup_{n\in \N}(q_{n}, q'_{n})  \subseteq \mathbb{R}$.\footnote{This is a special case of open sets in a separable metric spaces given by \citet[II.5.6]{Simpson2009}.}

Using such representations most theorems of classical real analysis can be proven in $\mathsf{Z}_2$.  But a motivation for the development of reverse mathematics was the discovery that many such theorems can be proven in \textsl{subsystems} such as the following:

\smallskip
\smallskip
\noindent $\bullet$ $\mathsf{RCA}_0$ consists of $\mathsf{PA}^-$ and the following restrictions to $\mathrm{Ind}_2$ and $L_2$-CA: 
\begin{itemize}
\item[- $\mathrm{Ind}$-$\Sigma^0_1$:]
Suppose $\varphi(n)\equiv (\exists m)\psi(n,m)$ for $\psi(n,m)$ contains only bounded quantifiers.  Then $[\varphi(0)\wedge (\forall n)(\varphi(n) \rightarrow \varphi(n+1))] \rightarrow (\forall m)\varphi(m)$.
\end{itemize}

\begin{itemize}
\item[- $\Delta^0_1$-$\mathrm{CA}$:]
Suppose $\varphi(n, k)$ and $\psi(n, m)$ contain only bounded quantifiers and \\ $(\forall n)[(\exists m)\varphi(n, m) \leftrightarrow (\forall k)\psi(n,k) ]$. Then $\exists X (\forall n)( n\in X\leftrightarrow (\forall k)\psi(n,k))$.
\end{itemize}
\smallskip
\smallskip
\noindent  $\bullet$  $\mathsf{WKL}_0$ consists of $\mathsf{RCA}_0$ with the $L_2$-formulation of \textsl{Weak K\"onig's Lemma}: 
\begin{itemize}
\item[- WKL:] Every infinite binary tree $T \subseteq 2^{\mathbb{N}}$ has a path.  
\end{itemize}

\smallskip
\smallskip

\noindent  $\bullet$  $\mathsf{ACA}_0$ consists of $\mathsf{RCA}_0$ with the \textsl{arithmetical comprehension schema}:

\begin{itemize}
\item[- ACA:] Suppose $\varphi(n,k)$ is a $L_2$-formula with number quantifiers and second-order parameters, but no set quantifiers.  Then $\exists X (\forall n)( n\in X\leftrightarrow (\forall k)\varphi(n,k))$
\end{itemize}

Recall that a set of natural numbers is \textsl{computable} just in case it is definable by both a $\Sigma^0_1$- and a $\Pi^0_1$-formula as required by the antecedent of $\Delta^0_1$-$\mathrm{CA}$.  It follows that $\mathsf{RCA}_0$ has as its minimal $\omega$-model the structure $\textsc{Rec}$ whose second-order domain is comprised of the computable sets.  In light of its minimal commitments, $\mathsf{RCA}_0$ is conventionally adopted as the \textsl{base theory} for reverse mathematics.  By contrast, a computability theoretic argument shows that there exist infinite recursive binary trees without recursive paths.  It thus follows that every $\omega$-model of $\mathsf{WKL}_0$ must contain non-computable sets.   The minimal $\omega$-model of $\mathsf{ACA}_0$ is the structure $\textsc{Arith}$ whose second-order domain is comprised of the arithmetically-definable sets -- i.e.\ those definable in the manner required by the antecedent of ACA.  This includes many non-computable sets -- e.g.\ Turing's \textsl{Halting problem} $K$.

It is possible to show in this manner that these three systems increase in strength. Indeed they form the first steps in sequence of canonical theories $\mathsf{RCA}_0 \subsetneq \mathsf{WKL}_0 \subsetneq \mathsf{ACA}_0 \subsetneq \mathsf{ATR}_0 \subsetneq  \Pi^1_1\text{-}\mathsf{CA}_0 \subsetneq \mathsf{Z}_2$ studied in the conventional development of reverse mathematics.  This framework has been used to classify hundreds of theorems of ordinary mathematics in the sense of showing that they are either provable in $\mathsf{RCA}_0$ or are equivalent over this theory to one of the other systems in this sequence.   It thus should not be surprising that the principles underlying the forms of the sorites paradox considered above fall within the scope of reverse mathematics. 

\subsection{Classificatory theorems}\label{penk2}

The mathematical results presented in this section illustrate two points on which we will build in \S\ref{S7}.    First, the three sorites arguments considered above  -- i.e.\ \textsl{discrete} (\S\ref{S2}), and \textsl{covering} (\S\ref{S5}), \textsl{continuous} (\S\ref{S3}) -- rest on increasing non-constructive principles. Second, this provides a principled reason for regarding the three arguments as distinct forms of the paradox.  

In regard to the first point, the formalization of mathematics in $\mathsf{RCA}_0$ has traditionally been regarded as compatible with the development of \textsl{constructive analysis} \citep[\S I.8]{Simpson2009}.  For instance within $\mathsf{RCA}_0$ it is possible to formalize proofs describing effective approximations of real numbers such as the non-enumerability of the reals.  Another example is the proof of original statement of H\"older's Theorem -- i.e.\ \textsl{every order Archimedean group is isomorphic to a subgroup of} $\langle \mathbb{R},<,+ \rangle$.  

\citet{Solomon1998a} showed that this statement is provable in $\mathsf{RCA}_0$.  By adapting his construction  
to accommodate the weaker axioms for $\mathcal{A}$ we can obtain the following:
\begin{proposition}  
\label{holderrca0}
Using the coding of the reals described above, the $L_2$-formalization of Theorem \ref{holderthm0} of \S\ref{S2} is provable in $\mathsf{RCA}_0$.
\end{proposition}

This result is significant for two reasons.  First, it exemplifies a presupposition of measurement theory that representation theorems ought to be accompanied by constructive procedures for associating magnitudes with real numbers.  Indeed the proof of H\"older's Theorem given by \citet{Krantz1968} describes a method for computing the measurement function $\phi: A \to \mathbb{R}$ via a process of successive bisecting approximations so that the value of $\phi(a)$ is given as sequences of rationals $q_0, q_1, \ldots$ with an effective modulus of convergence.\footnote{One first defines $N(b,a)$ to be the least $n \in \mathbb{Z}$ such that the $\circ$-concatenation of $n+1$ ``copies'' of $b$ exceed $a$ in the sense of $\prec$ (which exists by the Archimedean axiom).  Choosing a unit $u_0 \in A$, the other axioms entail that if $A$ is infinite there exist subunits $\langle u_k \rangle$ such that $u_{k+1} \circ u_{k+1} \preceq u_k$.  The value of $\phi(a)$ is then given by the series defined via $q_k := \frac{N(u_k,a)}{N(u_k,u)}$ which can be shown to converge with the same exponential modulus (\ref{mod2k}) used in the $L_2$-definition of reals.   This idea is implicit in H\"older's original presentation \citeyearpar{Holder1901a} and also employed by \citet{Solomon1998a}.   The proof is given in most detail by \citep[\S 2.2]{Krantz1971} who also present it as a paradigm of extensive measurement which may be adapted to prove other representation theorems.}   Second, it illustrates that while the linguistic formulation of discrete sorites arguments are often reliant on such representations, they do not require additional non-constructive principles beyond $\mathsf{RCA}_0$. 

Turning next to the continuous sorties, we begin by observing that it is possible to formulate a version of Open Induction as a single statement rather than a schema by quantifying over open sets directly rather than formulas which define them:
\begin{itemize}
\item[$\mathrm{OI}_4$:] For any open $O \subseteq [0,1]$, we have \\ $\forall y \in [0,1]((\forall z\in [0,1])(z<y\di z\in  O)\di y\in O) \rightarrow (\forall x\in [0,1])(x\in O))$
%For all open sets $O\subset [0,1]$ and $y\in [0,1]$ we have $(\forall z\in [0,1])(z<y\di z\in  O)\di y\in O$, then $(\forall x\in [0,1])(x\in O)$
%\item[\covexnumber{\OI$_3$}]  Any open set $O\subset [0,1]$ satisfies $(\forall x\in [0,1])(x\in O)$ if for all $y\in [0,1]$, we have $(\forall z\in [0,1])(z<y\di z\in  O)\di y\in O$.
\end{itemize}
The relation of this principle to the continuous sorites is clarified by the following.\footnote{Note that we do not include an $L_2$-formalization of $\mathrm{OI}_2$ as its antecedent contains the term $\sup(X)$ .   The equivalence of (\ref{mct})-(\ref{cauchoi}) and (\ref{ac}) makes clear that the existence of the suprema of sets of real numbers with relatively simple definitions is already sufficient to entail $\mathrm{ACA}$ over $\mathsf{RCA}_0$.}

\begin{theorem}\label{klop} The following statements are provably equivalent over $\mathsf{RCA}_0:$
\begin{enumerate}
\renewcommand{\theenumi}{\alph{enumi}}
\item The schema of Arithmetical Comprehension $\mathrm{ACA}$. \label{ac}
\item A bounded monotone sequence $(x_{n})_{n\in \N}$ of reals in $[0,1]$ converges to a limit.\label{mct}
\item If $(x_n)_{n \in \mathbb{N}}$ is a Cauchy sequence of reals, then $\lim_{n \rightarrow \infty} x_n$ exists. \label{cauch}
\item If $(x_{n})_{n\in \N}$ is a bounded sequence of reals, then $\sup(\{x_n : n \in \mathbb{N}\})$ exits.   \label{sdc}
\item If $C$ is a \textup{(}code for\textup{)} a closed subset of $[0,1]$, the infimum $\inf(C)$ exists.\label{inff}
\item If $O$ is a \textup{(}code for\textup{)} an open subset of $[0,1]$, the supremum $\sup(O)$ exists.\label{supp}
\item Open Induction in the form $\mathrm{OI}_{1}(\Psi)$ for $\Sigma_{1}^{0}$-formulas $\Psi(x)$. \label{cauchoi}
\item Open Induction in the form $\mathrm{OI}_{3}(\Psi)$ for $\Sigma_{1}^{0}$-formulas $\Psi(x)$. \label{tenk}
\item Open Induction in the form $\mathrm{OI}_{4}(\Psi)$.\label{plop}
\end{enumerate}
\end{theorem}
\begin{proof}
The equivalence between the first six items is well-known \citep[IV.2.11, III.2.2]{Simpson2009}, while the equivalence between  $\mathrm{OI}_{1}(\Psi)$ and  $\mathrm{OI}_{1}(\Psi)$ is immediate as $\Sigma_{1}^{0}$-formulas can be effectively converted into codes for open sets, and vice versa,    \cite[II.5.7]{Simpson2009}. %By symmetry, we can also replace `sup' by `inf' in item \eqref{supp} and `open' by `closed'. 
%WD: Replacing \inf C with \with \inf(C)
For the implication (\ref{ac}) $\Rightarrow$ (\ref{plop}), assume $\ACA_{0}$ and let $O$ be a code for an open set as in item \eqref{plop} -- i.e.\ we have for all $y\in [0,1]$ that 
\begin{example}\label{lenk}
$(\forall z\in [0,1])(z<y\di z\in  O)\di y\in O$
\end{example}
Note that fixing $y=0$ in \eqref{lenk}, we (trivially) obtain $0\in O$.   Suppose  $x_{0}\not \in O$ for some $x_{0}\in [0,1]$ -- i.e.\ $C:= [0,1]\setminus O$ is closed and non-empty.  Now consider $\inf(C)$ which satisfies $0<\inf(C)<1$ by assumption and in light of $0\in O$ and the assumption that $O$ is open. 
Now consider the instance of \eqref{lenk} for $y=\inf(C)$, i.e.\
\[
(\forall z\in [0,1])(z<\inf(C)  \di z\in O)\di \inf(C)\in O.
\]  
Since the antecedent holds, so does the consequent.   However, $\inf(C)\in O$ implies that $\{(\inf(C)-r,\inf(C)+r) \subset O$ for some $r>0$.
%WD: Replacing with $B(\inf C, r)\subset O$ for some $r>0$ with its defn
As this contradicts the definition of $\inf(C)$,  $(\forall x\in [0,1])(x\in O)$ must hold.  One derives item \eqref{tenk} in the same way.     

%One similarly proves item \eqref{tenk}, using 
%the fact that a $\Sigma_{1}^{0}$-formula defines a code for an open set (see \cite[II.5.7]{soa}).
%For the forward direction, note that ACA_0 is equivalent to 
%
%Òa non-empty RM-closed set has a sup (and inf)Ó.    
%
%Thus, if the consequent of open induction is false for some RM-open set O, 
%let C be the complement and consider inf C.  The latter real clearly makes 
%the associated antecedent of open induction false.  

%%\smallskip

For (\ref{plop}) $\Rightarrow$ (\ref{ac}), assume item \eqref{plop} and let $(x_n)_{n \in \N}$ be an increasing sequence of 
reals in $[0,1]$.  Consider the (code for an) open set defined as $O=\bigcup_{n\in \N} [0, x_n)$.  % and let $\Phi(x)$ be $(\exists n)( x<x_{n})$.  
%In case $O$ is all of $[0,1]$, we are done as then the limit of $(x_n)_ {n \in \N}$ is 1.  
If $1\not\in O$, the contraposition of $\mathrm{OI}_{4}$ yields that there is $y\in [0,1]$ such that \eqref{lenk} is false, i.e.\ 
\[
(\forall z\in [0,1])(z<y\di z\in  O)\wedge y\not\in O.
\]
Since $y$ is the limit of $(x_n)_{n \in N}$, the Monotone Convergence Theorem as in item \eqref{mct} follows; by the above, we obtain $\ACA_{0}$.  For the derivation of the latter via \eqref{cauchoi}, define $\Phi(x)\equiv (\exists n)(x<x_{n} )$ and use the prior proof. The proof based on \eqref{tenk} is similar. 
\end{proof}
This illustrates that the various forms of Open Induction already have strong consequences about set existence for relatively simple formula.\footnote{The fact that the restriction of  $\mathrm{OI}_{1}(\Psi)$ and $\mathrm{OI}_{3}(\Psi)$ already have this property for $\Sigma^0_1$-formulas is not  in light of the fact that the restriction of Arithmetical Comprehension to $\Sigma^0_1$-formulas is equivalent to the full schema \citep[III.1.3]{Simpson2009}.  More generally, the equivalence of parts (d) and (i) in Theorem \ref{klop} is a case of the observation that Open Induction is, in general, equivalent to the existence of suprema. The latter has traditionally been regarded as an \textsl{impredicative} feature of the real numbers.  This emerges more clearly when we move to higher-order arithmetic, in which case the assumption that Open Induction holds for a given class of formulas can again be shown to be equivalent to the supremum principle for this class over the base theory  $\mathsf{RCA}^{\omega}_0$ . The latter can then be used to show that open sets have second-order codes, yielding the considerably stronger second-order principle $\mathsf{ATR}_0$ by combining Theorems 4.3 and 4.5 in \citep{Normann2020}. \label{weylnote}} Turning now to the covering sorites, we obtain a similar classificatory theorem:
\begin{theorem}\label{klop2} The following statements are provably equivalent over $\mathsf{RCA}_0:$
\begin{enumerate}
\renewcommand{\theenumi}{\alph{enumi}}
\item Weak K\"onig's lemma $\mathrm{WKL}$. \label{wkl}
\item The Heine-Borel theorem for countable coverings. \label{hbi}
\item The Creeping Lemma $\mathrm{CL}(\Psi)$ for $\Sigma_{1}^{0}$-formulas $\Psi(x)$.\label{gonk}
\end{enumerate}
\end{theorem}
\begin{proof}
The equivalence between the first two items is well-known \citep[IV.1]{Simpson2009}.
Now assume the Creeping Lemma  as in item \eqref{gonk} and consider a countable covering $(O_{i})_{i\in \N}$ of $[0,1]$.  Define $\Psi(x)\equiv (\exists j)( [0, x] \subset \bigcup_{i<j} O_{i} )$, which can be shown to be $\Sigma_{1}^{0}$ relative to coding of open sets described above.
We also have $\Psi(0)$ and $\exists \mathcal{C}\mathrm{Cov}(\Psi,\mathcal{C} )$.  Using item $\mathrm{CL}(\Psi)$, we obtain $\Psi(1)$, i.e.\ $(O_{i})_{i\in \N}$ has a finite sub-covering of $[0,1]$.  Thus, item \eqref{hbi} follows as required.   

%%\smallskip
\enlargethispage{7ex}

To prove the Creeping Lemma as in item \eqref{gonk}, assume item \eqref{hbi} and let $(O_{i})_{i\in \N}$ be a covering of $[0,1]$.  Define $\Psi(x)$ be as above and note that this covering again witnesses  $\exists \mathcal{C}\mathrm{Cov}(\Psi,\mathcal{C})$.   Using item \eqref{hbi}, there is $j\in \N$ such that $[0,1]\subset \cup_{i<j}O_{i}$.  Since we assume $\Psi(0)$, we have $\Psi(z)$ for all $z \in O_{i_{0}}$ where $i_0$ is such that $0\in O_{i_{0}}$.  Since $[0,1]\subset \cup_{i<j}O_{i}$, there is $i_{1}<j$ with $O_{i_{0}}\cap O_{i_{1}}\ne \emptyset$, implying $\Psi(z)$ for all $z \in O_{i_{0}}\cup O_{i_{1}}$. Continuing in this fashion, one concludes $\Psi(z)$ for all $z \in \cup_{i<j}O_{i}$ after finitely many steps.   This reasoning can be formalized using $\Sigma_{1}^{0}$-induction. 
\end{proof}

WKL does not imply ACA over $\mathsf{RCA}_0$ \citep[VIII.2.12]{Simpson2009}.  The equivalence of WKL and CL thus illustrates that the continuous sorites requires \textsl{strictly stronger} mathematical principles than the covering sorites.   This substantiates our claim that they are distinct argument forms.\footnote{It is also natural to inquire further after the reverse mathematical status of Weber \& Colyvan's topological sorites for the space $[0,1]$ endowed with the Euclidean topology.   In this case the argument purports to show that if $U = \{x \in [0,1] : \Phi(x)\}$ and $V = \{x \in [0,1] : \neg \Phi(x)\}$ are both open and such that $U \cup V = [0,1]$, then if $U \neq \emptyset$ then in fact $U = [0,1]$ and thus also $\Phi(1)$.  This is  a reformulation of the  \textsl{connectedness of the unit interval} -- i.e.\ $[0,1]$ cannot be partitioned into disjoint open sets.  But the strength of this statement depends on how the definition is formalized -- e.g.\ it can be either at the level of $\RCA_{0}$ or $\ACA_{0}$, depending on how the complement of a set is given.  Since the covering sorites is at the intermediate level of $\mathsf{WKL}_0$, this is consistent with our claim that the covering sorites should be regarded as a distinct argument form.   But it also illustrates why it is more difficult to formalize the topological sorites in a robust manner.\label{topnote2}}  At the same time, the original proofs of Heine-Borel Theorem (such as the one given in \S\ref{S5}) all made use of principles implying ACA.  The separation of the two arguments thus also emerges as a distinctive contribution of reverse mathematics to the study of the sorites.\footnote{The observation that results based on compactness arguments often do not require set existence assumptions as strong as results relying on the continuity properties of the reals was itself not historically immediate. The systematization of such separations can thus be seen as a distinctive contribution of reverse mathematics to the arithmetization of analysis.  See \citep{Dean2017b}.} But the fact that CL \textsl{entails} WKL shows that the reasoning of the covering sorites still depends on non-constructive principles.   For while $\mathsf{WKL}_0$ is unlike $\mathsf{ACA}_0$ in that it does not prove the existence of \textsl{specific} non-computable sets, it still implies some such sets must exist \citep[XIII]{Simpson2009}.  We will now see how such computability-theoretic facts may be used to construct models in which the premises of both forms of the sorites hold but their putatively paradoxical conclusions fail.

\section{Resolutions}
\label{S7}

The continuous sorites and covering sorites rest on two sorts of major premises: 
\begin{enumerate}[(P1)]
\item  The schemas $\mathrm{Open}(\Psi)$, $\mathrm{Prog}^{[0,1]}(\Psi)$, $\mathrm{LCC}^{[0,1]}_{\sup}(\Psi)$, $\mathrm{LCC}^{[0,1]}_{\mathrm{seq}}(\Psi)$, and $\exists \mathcal{C}\mathrm{Cov}(\Psi)$ which we will continue to refer to collectively as \textsl{tolerance principles}.   
\item The schemes $\mathrm{OI}_1(\Psi)$, $\mathrm{OI}_2(\Psi)$, $\mathrm{OI}_3(\Psi)$, and $\mathrm{CL}(\Psi)$  which Theorems of \ref{klop} and \ref{klop2} provide further justification for calling \textsl{mathematical principles}.    
\end{enumerate}
The latter class are all conditional in form while the former serve as their antecedents in the arguments (\ref{csarg}$'$) and (\ref{cs}).  To block the reasoning it hence suffices to reject \textsl{either} (P1) \textsl{or} (P2).   We will refer to these as \textsl{Type-1} and \textsl{Type-2} responses.   

%The first alternative may be likened to the traditional response of supervaluationists and epistemicists to the sorites argument consider in \S \ref{S2}. But there is no apparent opportunity for invoking the second sort of response in such cases as the mathematical presuppositions of arguments like (\ref{fm}) seem to be more minimal.\footnote{The mathematical presuppositions of traditional sorites arguments are examined in greater detail in \citep{Dean2018a} -- e.g.\ in regard to the possibility of rejecting subschema of mathematical induction which are weaker than that subsumed within $\mathsf{RCA}_0$.   On the other hand,  we have argued in section \ref{S2} that the paradigmatic form of the conditional sorites argument for predicates like \textsl{short} implicitly relies on H\"older's Theorem.  But we have also seen in \S \ref{S5} that this is result is constructive in virtue of being provable in $\mathsf{RCA}_0$.}

To see how the prior analysis bears on these options consider:
\begin{corollary}\label{gohan}
There is a model $\mathcal{M} \models \RCA_{0}$ and $\Sigma^0_1$-formulas $\Sigma(x)$ and $\Pi(x)$ such that $\mathrm{OI}_1(\Sigma),\mathrm{OI}_3(\Sigma),\mathrm{OI}_4(\Sigma)$ and $\mathrm{CL}(\Pi)$ are false in $\mathcal{M}$.
\end{corollary}
\noindent 
In light of Theorems \ref{klop} and \ref{klop2}, we may take $\mathcal{M}$ to be the model $\textsc{Rec}$.  The formulas in question can then be obtained from traditional constructions known as \textsl{recursive counterexamples} -- i.e.\ computable objects  satisfying the antecedent of a classical theorem for which there does not exist a computable witness to its consequent.  For  since $\mathsf{RCA}_0$ proves the existence of objects satisfying the following definitions, it follows that the Monotone Convergence Theorem (item b) in Theorem \ref{klop}) and the Heine-Borel Theorem are both \textsl{false} when interpreted in the \textsl{computable reals} -- i.e.\ the sets $\mathbb{R}^{\textsc{Rec}}$ satisfying the $L_2$-definition \textnormal`$x$ \textsl{is a real number'} in the model $\textsc{Rec}$.

\begin{definition} A  \textnormal{Specker sequence} $(s_{n})_{n\in \N}$ is a computable, increasing, and bounded sequence of real numbers such that $s = \lim_{n \rightarrow \infty} s_n$ is not a computable real number.
\end{definition}

\begin{definition} A \textnormal{singular covering} consists of two computable sequences of $(s_{n})_{n\in \N}$ and $(t_{n})_{n\in \N}$ of computable real numbers such that $\mathcal{S} = \{(s_i,t_i) : i \in \mathbb{N}\}$ covers $[0,1]^{\textsc{Rec}}$ -- i.e.\ the class of computable real numbers in the unit interval -- but is such that $\sum_{i \in \mathbb{N}} |s_i - t_i| < 1$.   
\end{definition}

The relevant features of such objects (as reprised in the Appendix) are as follows:
\begin{enumerate}[(R1)]
\item From the perspective of the \textsl{classical continuum} $\mathbb{R}$ -- e.g.\ as characterized by Dedekind Completeness --  $\mathbb{R}^{\textsc{Rec}}$ contains ``gaps'' corresponding to non-computable real numbers.  If $(s_{n})_{n\in \N}$ is a Specker sequence and  $s = \lim_{n\in \mathbb{N}} s_n = \sup\{s_{n} : n\in \N\}$ then $s \in \mathbb{R} \setminus \mathbb{R}^{\textsc{Rec}}$ illustrates such a gap.   
\item If $\mathcal{S} = \{(s_i,t_i) : i \in \mathbb{N}\}$  is a singular covering,  then for each $x \in \mathbb{R}^{\textsc{Rec}}$  there exists an $i$ such that $x \in (s_i,t_i)$.  But since $\sum_{i \in \mathbb{N}} |s_i - t_i| < 1$, any attempt to cover all of $[0,1]^{\textsc{Rec}}$ must ``leave a gap'' in $[0,1]$.  As such, we cannot pass from $0$ to $1$ via any finite set of overlapping intervals $\mathcal{S}_0 = \{(s_{i_j},t_{i_j}) : j < k\}$.
\end{enumerate}
These properties are already suggestive of how the computable real numbers may be useful for providing a semantics for vague predicates.\footnote{Philosophers will find the use of the term ``gap'' evocative of views which posit ``truth values gaps'' as discussed below.  But this term also evokes talk of ``truth value \textsl{gluts}'' and ``overloaded boundaries'' as considered in approaches based on paraconsistent logic.   These proposals cannot be considered in detail here.   But it may still be noted that constructive analysis as also discussed below can be understood as an explicit (and historically prior) counterpoint to the paraconsistent form of analysis and topology envisioned in \citep{Weber2021}. \label{paranote}}   But it is not initially clear whether they motivate a Type-1 or a Type-2 response to our current paradoxes.  

Consider first a Specker sequence $(s_n)_{n \in \mathbb{N}}$ with $s = \lim_{n \rightarrow \infty} s_n$ and define  
\begin{example}
\label{ws}
$\Sigma(x) := 0 \leq x  \ \wedge \ \exists n(x < s_n)$
\end{example}
The set $S = \{x \in \mathbb{R}: \Sigma(x)\}$ will thus be an interval $[0,s)$. It may be confirmed that the premises of the continuous sorites $\Sigma(0)$, $\mathrm{Open}(\Sigma)$, $\mathrm{Prog}^{[0,1]}(\Sigma)$, $\mathrm{LCC}^{[0,1]}_{\sup}(\Sigma)$, and $\mathrm{LCC}^{[0,1]}_{\mathrm{seq}}(\Sigma)$ all hold in $\textsc{Rec}$.\footnote{For instance $\mathrm{Open}(\Sigma)$ holds because $(s_n)_{n \in \mathbb{N}}$ is increasing.   For $\mathrm{Prog}(\Sigma)$ consider an arbitrary $x \in \mathbb{R}^{\textsc{Rec}}$ and suppose that $\forall y(y < x \rightarrow \Sigma(y))$.   In this case, we cannot have $s < x$.   But since $s \not\in \mathbb{R}^{\textsc{Rec}}$, it thus follows that $x < s$ and thus since $s = \lim_{n \rightarrow \infty} s_n$, there is an $m$ s.t. $x < s_m$.   But then $\Sigma(x)$ holds by definition.  The arguments for $\mathrm{LCC}^{[0,1]}_{\sup}(\Sigma)$ and $\mathrm{LCC}^{[0,1]}_{\mathrm{seq}}(\Sigma)$ are more cumbersome since $\sup(S)$ and $\lim_{n \rightarrow \infty} s_n$  disguise  descriptions which are non-denoting in $\textsc{Rec}$.}   But since $s = \sup(S) < 1$, it also follows that $\neg \Sigma(1)$.   It also follows from the construction given in the Appendix that there is a total computable function $f(n)$ such that $f(n) =$ (the code of) $s_n$. $\Sigma(x)$ thus exemplifies a $\Sigma^0_1$-predicate for which each of the principles  $\mathrm{OI}_1,\mathrm{OI}_3$ and $\mathrm{OI}_4$ fail in $\textsc{Rec}$.  

A singular covering can be similarly employed to show how the premises of the covering sorites -- other than CL itself -- can hold in $\textsc{Rec}$ without entailing a contradiction.  For let $\mathcal{S} = \{(s_i,t_i) : i \in \mathbb{N}\}$ be a such a covering and define 
\begin{example}
\label{cs}
$\Pi(x) := \exists n([0,x] \subseteq \bigcup_{i < n} (s_i,t_i))$
\end{example}
It is easy to see that that $\Pi(0)$ and that $\forall i \forall x,y \in (s_i,t_i)(\Pi(x) \leftrightarrow \Pi(y))$ hold in $\textsc{Rec}$.\footnote{Suppose $x,y \in (s_i,t_i)$ and $\Pi(x)$.   Then there is $k$ is s.t. $[0,x] \subseteq \bigcup_{j < k} (s_j,t_j)$.   But since $[0,y] \subseteq \bigcup_{j < k} (s_j,t_j) \cup \{(s_i,t_i)\}$,  $[0,y] \subseteq \bigcup_{j < m+1} (s_j,t_j)$ for $m = \max(i,k)$.   Thus $\Pi(y)$ as well.}  $\mathcal{S}$ thus witnesses $\textsc{Rec} \models \exists \mathcal{C}\mathrm{Cov}(\mathcal{C},\Pi)$.  But since $\mathcal{S}$ does not admit a finite sub-cover in $\textsc{Rec}$, $\textsc{Rec} \not\models \Pi(1)$.   Thus $\Pi(x)$ is an example of $\Sigma^0_1$-predicates for which $\mathrm{CL}$ fails in  $\textsc{Rec}$.  Similarly if we let $s' = \sup \{x \in \mathbb{R}: \Pi(x)\}$ then we can see that $s' \in \mathbb{R} \backslash \mathbb{R}^{\textsc{Rec}}$ -- i.e.\ the supremum of the set of points reachable from $0$ via a finite sequence of overlapping intervals in $\mathcal{S}$ -- must itself be a non-computable real.

What one makes of these constructions may depend on  one's views about \textsl{both} the proper response to traditional discrete forms of the sorites and also the proper foundations for real analysis and its applications.   For instance if one maintains the classical view that the mathematical continuum is appropriately characterized by properties such as Dedekind Completeness, then one will presumably accept a mathematical theory at least strong as $\mathsf{ACA}_0$.  In this case, one will accept the mathematical principles (P2) and thus presumably be predisposed to reject one of the principles under (P1) in the manner of a Type-1 response to the sorites.  

But even in this case, properties (R1) and (R2) of $\mathbb{R}^{\textsc{Rec}}$ may still play a role.  For recall that the rejection of tolerance principles has traditionally been regarded as an inevitable -- although potentially undesirable -- consequence of theories of vagueness known as \textsl{supervaluationism} (e.g.\ after \citealp{Fine1975}) and \textsl{epistemicism} (e.g.\ after \citealp{Williamson1994}). In particular, both supervaluationists and epistemicists hold that principles like (\ref{tol1}) will have false instances.  

Supervaluationists, for instance, hold that a statement $\Delta$ containing a vague predicate $\Phi_0(x)$ with field $A$ should be counted as true only if it is in fact \textsl{supertrue}.  If we let $P^+ \subset A$ be the \textsl{definite extension} of $\Phi_0(x)$, this is to say that $\Delta$ is satisfied in all interpretations in which $\Phi_0(x)$ corresponds to an \textsl{admissible precisification} $P \supseteq P^+$ -- i.e.\ a means of exhaustively deciding \textsl{indefinite} (or \textsl{borderline}) cases so that $P \cup \overline{P} = A$ and also $P \cap P^- = \emptyset$ where $P^-$ is the \textsl{definite anti-extension} $P^-$ of $\Phi_0(x)$.   But of course each such precisification $P$ will yield an instance $b \in A$ making the antecedent of principles like (\ref{tol1}) true but its consequent false.  

But now observe that the predicate $\Sigma(x)$ will determine an interval $[0,s)$ whose boundary occurs at a non-computable real $s \in \mathbb{R} \backslash \mathbb{R}^{\textsc{Rec}}$.   On the other hand, it seems reasonable to maintain that the definite extension $S^+$ of $\Sigma(x)$ must be fixed by a sequence of observations which can be carried out so as to determine an interval of the form $[0,t]$ or $[0,t)$ where $t < s$ is a \textsl{computable} real number -- e.g.\ as resulting from a measurement procedure with a computable modulus of convergence similar to the envisioned construction
in the proof of Theorem \ref{holderthm0} described above.\footnote{What is conventionally meant  by saying  $\Phi_0(a)$ is \textsl{definitely true} is that $a$ is \textsl{not} a borderline of $\Phi_0(x)$.  In the current context, this suggests that we may determine if $\Phi(x)$ holds for $\phi(a)$ by some fixed approximation to its actual value -- e.g. by examining $k$ digits in its decimal expansion.   Of course, we may not know the value of $k$ in advance or even if the definite extension $S^+$ of $\Phi(x)$ is an open or closed interval.  But since such a value is presumably fixed independently of $a$  (e.g. by our perceptual thresholds),  it still seems as if $S^+$ ought to have computable endpoints.}

Putting aside the question of whether non-computable real numbers \textsl{exist}, there is still the further question of whether $S \supseteq S^+$ must be acknowledged as an \textsl{admissible} precisification of the definite extension of $\Sigma(x)$.   For suppose there is a rationale for \textsl{excluding} such sets from the  class of precisification which must be considered in a supervaluationalist semantics.   It will then follow that even though the relevant instances of the tolerance principles falling under (P1) are \textsl{false} in $\mathcal{N}$, they will be \textsl{supertrue} when we restrict attention to precisifications corresponding to intervals whose endpoints exist in the model $\textsc{Rec}$.  This conclusion might in turn be embraced by supervaluationists who must otherwise provide an account of \textsl{why} the intuitions underlying tolerance principles are misleading.\footnote{In Fine's original formulation, the notion of an \textsl{admissible precisification} is ``officially'' declared to be primitive \citeyearpar[p. 272]{Fine1975}.  However ``unofficially'' it is taken to be induced by an ``\textsl{appropriate} [specification of cases] in accordance with the intuitively understood meanings of the predicate'' (p. 268).  It is thus at least open to supervalationists to accord computability-theoretic considerations a role in a supervaluational semantics understood (e.g.) as a theory of the precisifications which natural language users are able to consider in applying the relevant expressions.  }       

It also seems possible for epistemicists to similarly avail themselves of the features of properties (R1) and (R2) of $\mathbb{R}^{\textsc{Rec}}$.  For recall that such theorists are also compelled to reject tolerance principles because they wish to adhere to classical (in particular \textsl{non-supervaluational}) semantics. This in turn requires positing a sharp boundary between the extension of a vague and its anti-extension.  Epistemicists then attempt to mitigate the necessity of abandoning tolerance intuitions by offering an auxiliary account of why we cannot \textsl{know} where such a boundary is located.  

In the context of a discrete sorites like (\ref{fm}), the boundary between the extension of $\Phi_0(x)$ and its anti-extension must coincide with a fixed object $a_i$ in a sequence $a_0, \ldots,a_{n}$.  Taking into account that the measurement procedures which we have suggested in \S\ref{S2} are required to formalize the relevant argument, it seems likely the corresponding numerical boundary $\phi(a_i)$ will be given by a \textsl{rational number} $q_i \in \mathbb{Q}$.\footnote{This is so because the sequence $a_0, \ldots,a_{n}$ must presumably be found by a uniform process similar to using a ruler or pan balance relative to a fixed unit $c$.  But also note that all that is required by such a sequence is that the difference between the magnitudes $a_j$ and $a_{j+1}$ is less than some appropriate just noticeable difference  $d$.   Once a measurement function $\phi : A \rightarrow \mathbb{R}$ is in place, it suffices for the argument to represent the elements $a_j$ by rational multiples of $\phi(c)$ which can itself be normalized to the value of the form $10^{-n}$ as in the example considered in \S\ref{S2}.}  But since such a value is finitely specifiable, it is then difficult to substantiate the epistemicist's claim that its value cannot be known \textsl{even in principle}.

But of course the situation is again different in the case of the continuous and covering sorites wherein the relevant boundaries are determined by a real number. The constructions under consideration additionally call attention to the fact that such a value may be non-computable.  But when we examine the methods by which such values might be determined, it seems there is indeed even an \textsl{in practice} sense in which we cannot come to learn their exact values.  This in turn suggests that epistemicists may also find it useful to employ recursive counterexamples to illustrate the sense in which they deem such boundaries to be unknowable.\footnote{In Example \ref{sse} $s := \lim_{n \rightarrow \infty} s_n$ is given so that the current value $s_i$ may ``jump'' by as much as $\frac{1}{2}$ arbitrarily late in its construction (e.g.\ if $0$ is enumerated into $A$ by the computable function $f$ only for a large input $n$).    This is in contrast to a procedure which determines a real number with a computable modulus of convergence such as is given by Theorem \ref{holderthm0} -- i.e.\ so that at any given stage we know how close our current approximation is to ultimate value.  In light of this, a Specker sequence might be taken to illustrate Williamson's claim that ``meaning may supervene on use in an unsurveyably chaotic way'' \citeyearpar[p. 209]{Williamson1994}.   More generally, such constructions provide a means of making good on the persistent hopes of epistemicists to explain the unknowability they posit for the boundaries of vague predicates via analogies with mathematical incompleteness or ``absolute undecidability'' (see, e.g, \citealp[\S 7.4]{Williamson1994}). \label{compnote}}

Turning now to Type-2 responses, readers from logic will be aware that such constructions were originally presented within foundational programs which seek to \textsl{deny} various features of the classical continuum.  In this context Specker sequences and singular covers serve as \textsl{weak counterexamples} to continuity properties of $\mathbb{R}$ or to its compactness -- \textsl{if} we accept such principles, \textsl{then} these constructions show that we must also accept non-constructive principles.  Within the historical development of constructivism, the availability of independent arguments \textsl{against} such principles is typically presupposed.\footnote{For instance the  existence of the supremum of a Specker sequence or the compactness of the unit interval are inconsistent with $\mathrm{CT}_0$ -- i.e.\ the intuitionistic form of Church's Thesis asserting that every function on $\mathbb{N}$ is computable (see, e.g., \citealp[pp. 268-309]{Troelstra1988}).}  Practicing constructivists have shown little interest in the interpretation of non-mathematical language.  But readers from philosophy will be aware that there is a related tradition of attempting to reply to sorites arguments by \textsl{rejecting} -- or at least failing to assert -- that vague predicates are \textsl{bivalent}.   

A precedent can be traced to \citet{Bernays1935aaa} who observed a contradiction could be avoided by abstaining from applying the law of the excluded middle to the evidently soritical predicate \textsl{feasible natural number}  -- i.e.\ one up to which we can count in practice.   This example was famously repurposed by \citet{Dummett1975} as part of his critique of strict finitism.\footnote{See \citep{Dean2019} for an account of the intellectual transmission.} Building on Dummett's broader view that a \textsl{theory of meaning} for natural language ought to be based on intuitionistic logic, \citep{Wright2019a} reiterated the proposal that semantics for vague predicates should be based on an intuitionistic understanding of the logical connectives. This proposal was recently put into a more precise form by \citep{Bobzien2020}. 

\enlargethispage{5ex}

These considerations are relevant here in light of our prior characterization of $\mathbb{R}^{\textsc{Rec}}$ as containing ``gaps'' corresponding to non-computable real numbers.  This naturally invokes comparisons not only to failures of excluded middle but also to approaches which posit \textsl{truth value gaps} for vague predicates.  Indeed the basic analogy on which the aforementioned theorists build is the claim that bivalence can only be asserted for \textsl{decidable} predicates.  For on the one hand, mathematical examples motivating the non-assertability of $\forall x(\Phi(x) \lor \neg \Phi(x))$ relative to the \textsl{proof} (or BHK) \textsl{interpretation} of the logical connectives make use of cases in which $\forall x \Phi(x)$ is undecided (or even \textsl{undecidable}).\footnote{See, e.g., \citep[pp. 10-16]{Troelstra1988}.}  On the other, philosophical accounts often characterize vague predicates $\Phi_0(x)$ as ones for which there exist borderline cases $a \in A$ such that the truth value of $\Phi_0(a)$ is difficult (or even \textsl{impossible}) to decide.

It would thus seem that theorists inclined towards constructivist approaches to vagueness may find Type-2 responses congenial.  For if $(s_n)_{n \in \mathbb{N}}$ is a Specker sequence with limit $s$, then the $L_2$-formalization of $\Sigma(s)$ is \textsl{formally undecidable} in $\mathsf{RCA}_0$.\footnote{The term $\lim_{n \to \infty} s_n$ need not denote in a given model of $L_2$.  But it can be seen that the assertion that $\Sigma(x)$ is satisfied by its classical denotation $s \in \mathbb{R} \backslash \mathbb{R}^{\textsc{Rec}}$ holds in $\textsc{Rec}$ (as the relevant set existence presupposition fails) but fails in the full model $\mathcal{N}$ in virtue of the definition of $s$.} If we then apply the proof interpretation to $\Sigma(s) \lor \neg \Sigma(s)$, this provides a basis for failing to assert $\forall x(\Sigma(x) \lor \neg \Sigma(x))$.\footnote{\citep{Bobzien2020} suggest that vague language should be interpreted into the language of modal logic via a version of the G\"odel embedding of intuitionistic logic into $\mathsf{S4}$.   On this interpretation, $\Sigma(s) \lor \neg \Sigma(s)$ is mapped to a sentence equivalent to  $\Box \Sigma(s) \lor \Box \neg \Sigma(s)$.  While G\"odel originally suggested reading  $\Box \varphi$  as ``it is informally provable that $\varphi$'',  Bobzien \& Rumffit's proposed gloss for vague language is ``it clear that $\varphi$'' (understood as abbreviating ``$\varphi$ and it is not borderline whether $\varphi$''). The current observation is thus that a further analysis of the sort of provability-cum-clarity at issue can be obtained by considering the interpretation of $\Box$ as provability in $\mathsf{RCA}_0$.  In this case $\Box \Sigma(s) \lor \Box \neg \Sigma(s)$ \textsl{demonstrably fails} without the need for other assumptions.   (For well-known reasons, the corresponding modality will not satisfy the T axiom of $\mathsf{S4}$.  But the evidential reading of ``clarity'' on which the constructivist approach to vagueness is presumably based does not on its own seem to prefer informal to formal provability.)} But this seems tantamount to holding open that either of its disjuncts might be true.  But in doing so, one must also presumably acknowledge the existence of models like $\textsc{Rec}$ (in addition to the full model $\mathcal{N}$) that can serve as felicitous interpretations of vague predicates despite the fact that they fail to satisfy the principles under (P2).  On its own, however, this does not recommend intuitionistic logic as a medium for reasoning about them.\footnote{This is so because $\textsc{Rec}$ -- like all models of $L_2$ -- satisfies classical logic.    The reason why the arguments of the continuous and covering are blocked in this structure is thus due to the failure of a \textsl{mathematical principle} rather than a \textsl{logical law}.  This illustrates how the discernment of additional structure in the premises of sorites arguments opens novel means of reply which are not bound up with questions about the use of non-classical logic.}

\section{Conclusion} \label{S8}
\enlargethispage{2ex}

The foregoing illustrates several connections between the reverse mathematical analysis we have provided and established philosophical approaches to vagueness.  For instance if foundational programs like constructivism are viewed as providing accounts of \textsl{applied mathematics}, then they can be understood as providing principled reasons for adopting a Type-2 response without a detour through meaning-theoretic considerations which motivate the proposals just described.\footnote{The situation is complicated by the fact that some of the principles which have been taken to describe the intuitionistic continuum (e.g.\ the Fan Theorem) are incompatible with $\mathrm{CT}_0$.  Historically recognized forms of constructivism may thus interact in different ways with the continuous and covering sorites.  This can be at least partially accounted for in the setting of reverse mathematics by considering $\omega$-models $\mathcal{M} \models \mathsf{WKL}_0$ which do not also satisfy ACA (e.g.\ the  \textsl{low subsets} of $\mathbb{N}$) and noting that $\mathsf{WKL}_0$ has been taken to describe a standpoint intermediate between constructive analysis and the full Brouwerian continuum (see, e.g., \citealp[I.12, VIII.2]{Simpson2009}).  But we postpone detailed consideration of the application to vagueness to another occasion.}  But another orientation is suggested by the analysis of results from measurement theory such as Proposition~\ref{holderrca0}.  We have suggested that this is required to formulate discrete sorites arguments for concrete or perceptual continua.  But we can now see this is compatible with viewing the computable reals $\mathbb{R}^{\textsc{Rec}}$ as an adequate representation.   

On the other hand, the continuous and covering sorites emerge as natural generalizations whose motivation is arguably implicit in that of traditional discrete forms.  But as Theorems \ref{klop} and \ref{klop2} attest, to arrive at a formal contradiction in these cases we must engage in reasoning reliant on mathematical principles whose adoption requires us to abandon $\mathbb{R}^{\textsc{Rec}}$ as a model of the envisioned structures.   This opens the door to another sort of reaction which requires neither advancing constructive scruples nor navigating the exigencies of views like supervaluationism and epistemicism.  In particular, such results can be taken to illustrate that reasoning of the paradoxes can be resisted on the basis of implicitly relying on an \textsl{overly idealized} conception of the continua involved in everyday reasoning about vague predicates.  This would also appear to counsel a Type-2 response -- but in this case one directed towards the applications of analysis in certain domains rather than its foundations.  

\appendix

\section{Recursive counterexamples}
\enlargethispage{2ex}

\begin{exa}[Specker sequences\footnote{See \citep[\S IV]{Beeson1985} for the historical origins of these constructions.}]
\rm
\label{sse}
Let $f:\N \rightarrow \N$ be a recursive function which enumerates without repetitions some non-recursive set $A$ -- e.g.\ a 1-1 enumeration of the halting problem $K$.  Define a computable sequence of rationals as follows:
\begin{example}\label{speck}
$\textstyle s_{n}:= \sum_{m\leq n}  \frac{1}{2^{f(m)+1}}$
\end{example}
Clearly, $(s_{n})_{n\in \N}$ is increasing and bounded.   Thus classically $s:=\lim_{n \rightarrow \infty} s_n$ determines a real number.  It is now easy to see that
\begin{example}\label{bink}
$\textstyle
(\exists m)(f(m)=k) \leftrightarrow (\forall n)(|s-s_{n}|<\frac{1}{2^{k}} \rightarrow (\exists i\leq n)(f(i)=k))$
\end{example}
where the left-hand side expresses that $k$ is the set enumerated by $f(x)$.  $s$ can be coded by a set  $S \subseteq \mathbb{N}$ in the manner describes in \S\ref{S6}.1.  Using \(\ref{bink}\) we can then see that if it were possible to effectively decide membership in $S$, then it would also thereby  be possible to effectively decide membership in $A$.     
\end{exa}

\begin{exa}[Singular covers]
\rm
\label{sce}

For motivation, let $(q_{n})_{n\in \N}$ be an enumeration (without repetitions) of the rationals in the unit interval and let $B(x,r) :=(x-r,x+r)$ -- i.e.\ the open ball centered at $x$ with radius $r$. Note that for fixed $m \in \mathbb{N}$, the covering $\bigcup_{n\in \N}B(q_{n }, \frac{1}{2^{n+m+1}})$ has total length at most $\sum_{n=0}^{\infty}|B(q_{n}, \frac{1}{2^{n+m+1}})| = \frac{1}{2^{m}}\sum_{n=0}^{\infty} \frac{1}{2^{n}} = \frac{1}{2^{m}} < 1$.  Constructing a singular cover $\mathcal{S}$ similarly requires that we cover all computable reals $z \in [0,1]^{\textsc{Rec}}$ by effectively giving a sequence containing approximations to $z$ together with a radius sufficient to guarantee they cover $z$.  To this end, let $p_{1}, \dots, p_{m}, \dots$ be an effective enumeration without repetitions of all indices $e\in \N$ to partial computable functions such that ${\phi_{e}}(i)$ is defined for all inputs $i\leq 2^{e+5}$.  Let $\eta:\mathbb{N} \di \mathbb{Q}\cap [0,1]$ be an effective coding of rationals via natural numbers.   Now define $\mathcal{S}$ as the sequence $(A_{n})_{n\in \N}$ of open balls such that $A_{m} := B(\eta\big(\phi_{p_{m}}(2^{p_{m}+4})), \frac{1}{2^{p_{m}+3}})$.   A real $z \in [0,1]^{\textsc{Rec}}$
computed by $\phi_{e}$ then satisfies $|z-\eta(\phi_{e}(2^{e+4}))|<\frac{1}{2^{e+4}}$ by definition, i.e.\ $z \in A_{e}$ as required. 

\end{exa}

%{\footnotesize
%\bibliographystyle{rsl}
%\bibliography{wdeanMain}}

\end{document}